\documentclass[11pt, reqno, english]{amsart}  
\usepackage[utf8]{inputenc}
\usepackage[T1]{fontenc}
\usepackage{amsmath,amsthm}
\usepackage{amsfonts,amssymb}
\usepackage{url}
\usepackage{mathtools}  
\usepackage[colorlinks=true,urlcolor=blue,linkcolor=red,citecolor=magenta]{hyperref}
\usepackage{enumerate, paralist}
\usepackage{tikz-cd}
\usepackage[left=1in,right=1in,top=1in,bottom=1in]{geometry}
\numberwithin{equation}{section}

\theoremstyle{plain}
\newtheorem{theorem}{Theorem}[section]
\newtheorem{lemma}[theorem]{Lemma}

\newtheorem{proposition}[theorem]{Proposition}

\theoremstyle{definition}

\newtheorem{remark}[theorem]{Remark}

\newcommand{\R}{\mathbb{R}}

\newcommand{\Z}{\mathbb{Z}}
\newcommand{\F}{\mathcal{F}}
\newcommand{\G}{\mathcal{G}}
\newcommand{\KG}{\mathrm{KG}}
\DeclareMathOperator{\conv}{\mathrm{conv}}

\DeclareMathOperator{\Null}{\mathrm{Null}}
\DeclareMathOperator{\rank}{\mathrm{rank}}

\begin{document}

\title[Transversal generalizations of hyperplane equipartitions]{Transversal generalizations of hyperplane equipartitions}



\author{Florian Frick}
\address[FF]{Dept.\ Math.\ Sciences, Carnegie Mellon University, Pittsburgh, PA 15213, USA}
\email{frick@cmu.edu} 

\author{Samuel Murray}
\address[SM]{Dept.\ Math.\ Sciences, Carnegie Mellon University, Pittsburgh, PA 15213, USA}
\email{slmurray@andrew.cmu.edu} 

\author{Steven Simon}
\address[SS]{Dept.\ Math.\, Bard College, Annandale-on-Hudson, NY 12504}
\email{ssimon@bard.edu} 

\author{Laura Stemmler}
\address[LS]{Dept.\ Math.\ Sciences, Carnegie Mellon University, Pittsburgh, PA 15213, USA}
\email{lstemmle@andrew.cmu.edu}

\thanks{FF was supported by NSF grant DMS 1855591, NSF CAREER Grant DMS 2042428, and a Sloan Research Fellowship.}


\begin{abstract} 
\small

The classical Ham Sandwich theorem states that any $d$ point sets in $\R^d$ can be simultaneously bisected by a single affine hyperplane. A generalization of Dolnikov asserts that any $d$ families of pairwise intersecting compact, convex sets in~$\R^d$ admit a common hyperplane transversal. We extend Dolnikov's theorem by showing that families of compact convex sets satisfying more general non-disjointness conditions admit common transversals by multiple hyperplanes. In particular, these generalize all known optimal results to the long-standing Gr\"unbaum--Hadwiger--Ramos measure equipartition problem in the case of two hyperplanes. Our proof proceeds by establishing topological Radon-type intersection theorems and then applying Gale duality in the linear setting. For a single hyperplane, this gives a new proof of Dolnikov's original result via Sarkaria's non-embedding criterion for simplicial complexes.

\end{abstract}

\date{\today}
\maketitle

\section{Introduction}

The starting point for our investigation is the classical Ham Sandwich theorem conjectured by Steinhaus and proved by Banach (see~\cite{BZ04}), a result which is at the origin of topological methods in discrete geometry:

\begin{theorem}
\label{thm:hamsandwich}
	Let $X_1, \dots, X_d$ be finite subsets of $\R^d$. Then there exists an affine hyperplane~$H$ such that both of the corresponding open halfspaces $H^+$ and $H^-$ contain no more than half the points of each~$X_i$.
\end{theorem}

The hyperplane $H$ given by Theorem~\ref{thm:hamsandwich} simultaneously bisects each set~$X_i$. This result is often stated for finite Borel measures instead of point sets, which follows from the formulation above by approximation. Conversely,  Theorem~\ref{thm:hamsandwich} can be derived from the version for Borel measures via a standard compactness argument.

We recall two generalizations of the Ham Sandwich theorem. First, Dolnikov gave a combinatorial criterion for not necessarily balanced hyperplane partitions of point sets. This recovers the Ham Sandwich theorem in the balanced case. Second, generalizations for equipartitions by multiple hyperplanes have been proven. In the same way that Dolnikov's result extends the classical Ham Sandwich theorem, we generalize equipartition results by multiple hyperplanes to the not necessarily balanced setting.

\subsection{Two generalizations of the Ham Sandwich theorem}
Dolnikov~\cite{Do92} generalized Theorem~\ref{thm:hamsandwich} to hyperplane piercings of families of compact convex sets. Recall that a set family $\F$ is said to be \emph{intersecting} if any two members of the family have non-trivial intersection. While we state Dolnikov's theorem below for families of polytopes (i.e., convex hulls of finite point sets in~$\R^d$), a standard approximation argument ensures the result for arbitrary compact convex sets. 

\begin{theorem}
\label{thm:dolnikov}
	Let $\F_1, \dots, \F_d$ be finite families of polytopes in $\R^d$, where each family is intersecting. Then there exists an affine hyperplane~$H$ that intersects each polytope from each~$\F_i$.
\end{theorem}

This generalizes Theorem~\ref{thm:hamsandwich}: Given finite sets ${X_1, \dots, X_d \subset \R^d}$, for each $X_i$ consider the corresponding intersecting family $\F_i = \left\{\conv A \ : \ A \subset X_i, \ |A| > \frac{|X_i|}{2}\right\}$. By Theorem~\ref{thm:dolnikov} there is an affine hyperplane~$H$ that intersects all sets from each~$\F_i$. This hyperplane must bisect each~$X_i$; otherwise, there is some~$i$ such that one of the open halfspaces $H^\pm$ contains more than half the elements of~$X_i$, say $|H^+ \cap X_i| > \frac12|X_i|$. Then $\conv (H^+\cap X_i) \in \F_i$ is disjoint from~$H$, contradicting that $H$ intersects all polytopes in every~$\F_i$. Theorem~\ref{thm:dolnikov} allows for the (unbalanced) case that $\F_i$ consists of polytopes, where the inclusion-minimal polytopes do not all have the same number of vertices.

Another generalization of the Ham Sandwich theorem concerns equipartitions by more than one hyperplane. Given finite sets $X_1, \dots, X_m \subset \R^d$, we say that $k$ (not necessarily distinct) affine hyperplanes $H_1, \dots, H_k \subset \R^d$ form an \emph{equipartition} if each of the (not necessarily distinct) $2^k$ open orthants $H_1^\pm \cap \dots \cap H_k^\pm$ contains at most $\frac{|X_i|}{2^k}$ of the points of each~$X_i$. The smallest dimension~$d$ such that any $m$ finite sets $X_1, \dots, X_m \subset \R^d$ admit an equipartition by $k$ affine hyperplanes is denoted by~$\Delta(m,k)$. Thus Theorem~\ref{thm:hamsandwich} states that $\Delta(m,1) \le m$, and this upper bound is easily seen to be tight. As with Theorem~\ref{thm:hamsandwich}, the problem of determining $\Delta(m,k)$ can be stated for more finite Borel measures more generally, and again approximation arguments show that the two problems are equivalent.  

For multiple hyperplanes, Avis~\cite{Av84} and Ramos~\cite{Ra96} observed the general lower bound $\Delta(m,k) \ge \lceil\frac{2^k-1}{k}m \rceil$ by placing the sets $X_1, \dots, X_m$ along the moment curve. This lower bound is conjectured to be optimal for all $m$ and~$k$.  The current best upper bound for general~$k$ is due to Mani-Levitska, Vre\'cica, and \v Zivaljevi\'c~\cite{ MLVZ06}: 

\begin{theorem}
\label{thm:hyperplane-bound}
    Let $m=2^p+q$, where $p\ge 0$ and $0\le q<2^p$. 
    Then $\Delta(m,k) \le 2^{k+p-1}+q$.
\end{theorem}

While recovering Theorem~\ref{thm:hamsandwich} for arbitrary~$m$, the upper bound of Theorem~\ref{thm:hyperplane-bound} coincides with the lower bound above only when $k=2$ and $m+1$ is a power of two. Additional matching upper bounds have been established for $k=2$ when $m$ or $m-1$ is a power of two, for $k=3$ when $m \in \{1,2,4\}$~\cite{BFHZ16}, and for no values of $m$ when $k\geq 4$; see~\cite{BFHZ18} for a survey of known results. We record the current situation as follows:

\begin{theorem}
\label{thm:hyperplane-exact}~ 
\begin{compactenum}[(i)]
        \item $\Delta(2^s-1,2) = 3\cdot 2^{s-1} -1$ for all integers $s\geq 1$
        \item $\Delta(2^s+1,2) = 3\cdot 2^{s-1} +2$ for all integers $s\geq 2$ 
        \item $\Delta(2^s,2) = 3\cdot 2^{s-1}$ for all integers $s\geq 1$
        \item $\Delta(1,3)=3$
        \item $\Delta(2,3)=5$
        \item $\Delta(4,3)=10$
    \end{compactenum}
\end{theorem}

\subsection{Statement of main results}
For a finite family $\F$ of non-empty sets and an integer $r\ge 2$, denote by $\KG^r(\F)$ the $r$-uniform \emph{Kneser hypergraph} of~$\F$, that is, the hypergraph with vertex set~$\F$ and a hyperedge $\{A_1, \dots, A_r\}$ for every collection of pairwise disjoint $A_1, \dots, A_r \in \F$. In particular, $\KG^2(\F)$  is the Kneser graph $\KG(\F)$ of ~$\F$. The \emph{chromatic number} ~$\chi(\KG^r(\F))$ of this hypergraph is the smallest number of colors required to color the vertices so that no hyperedge is monochromatic.  As usual, we denote this by $\chi(\KG(\F))$ when $r=2$. Note that $\chi(\KG^r(\F)) \leq m$ means precisely that there is a partition $\F = \F_1 \cup \dots \cup \F_m$ such that no $r$ sets from any $\F_i$ are pairwise disjoint. We say that a collection of $k$ hyperplanes \emph{pierces} a given family $\F$ of non-empty sets of $\R^d$ if the union of these hyperplanes intersects each set from~$\F$. We may now state our main results. 

\begin{theorem}
\label{thm:upper-transversal}
    Let $m=2^p+q$, where $p\ge 0$ and $0\le q<2^p$, and let $d=2^{k+p-1}+q$. If $\F$ is a finite family of polytopes in $\R^d$ such that $\chi(\KG^{2^k}(\F)) \le m$, then there exist $k$ affine hyperplanes in~$\R^d$ which pierce~$\F$. 
\end{theorem}

In the same way that Dolnikov's theorem implies the Ham Sandwich theorem, Theorem~\ref{thm:upper-transversal} generalizes Theorem~\ref{thm:hyperplane-bound}; see Theorem~\ref{thm:implications} for the proof of this implication. Note that the $k=1$ case of Theorem~\ref{thm:upper-transversal} is Dolnikov's theorem.

\begin{theorem}
\label{thm:two-transversal}
	Let $m=2^s+t$, where $s\geq 1$ is an integer and $t\in \{-1,0, 1\}$. Let $d = \lceil\frac32m\rceil$. If $\F$ is a finite family of polytopes in~$\R^d$ such that $\chi(\KG^4(\F))\le m$, then there exist two affine hyperplanes in~$\R^d$ which pierce $\F$.
\end{theorem}

Theorem~\ref{thm:two-transversal} generalizes Theorem~\ref{thm:hyperplane-exact}, parts (i), (ii), and~(iii). As with Theorem~\ref{thm:dolnikov}, the above theorems can be extended to arbitrary compact convex sets by approximation. The first non-trivial exact bound for hyperplane mass equipartitions for $k=2$ hyperplanes is due to Hadwiger~\cite{Ha66}, who showed that $\Delta(2,2) = 3$. In this instance, Theorem~\ref{thm:two-transversal} states the following: Given two families of polytopes in~$\R^3$, neither of which contains four pairwise disjoint polytopes, there are two affine planes whose union intersects each polytope from each family. 

    Our final main result gives a transversal extension of $\Delta(1,3)=3$: 
    
\begin{theorem}
\label{thm:three-transversal} 
    Let $\F$ be a finite family of polytopes in $\R^3$ such that $\chi(\KG^8(\F))=1$ (i.e., no eight polytopes are pairwise disjoint). Then ~$\F$ can be pierced by three affine hyperplanes in $\R^3$.
\end{theorem}

\subsection{Proof scheme: Topological Radon-type results for multiple Radon pairs}
Our proof approach for hyperplane partition and transversal results appears to be new. Here we explain this approach and give an overview of the proof scheme. Ultimately, we derive our main results from certain non-embeddability results for simplicial complexes, which we collect here since they might be of independent interest.

Gale duality translates between hyperplane bisection results for $n$ points in $\R^d$ and intersection results for convex hulls of $n$ points in $\R^{n-d-1}$; see Section~\ref{sec:dolnikov} for details. Results that prescribe constraints for disjoint subsets $A$ and $B$ of a set $X\subset \R^d$ whose convex hulls intersect are referred to as Radon-type results. The classical Radon theorem~\cite{Ra21} states that any $X\subset \R^d$ of size $d+2$ admits a partition $A \sqcup B = X$ such that $\conv A \cap \conv B \ne \emptyset$. Such partitions $A \sqcup B$ are called \emph{Radon partitions}. A rather general result asserting the existence of Radon partitions that satisfy certain constraints is the following theorem due to Sarkaria~\cite{Sa90, Sa91, Ma08}:

\begin{theorem}
\label{thm:sarkaria}
	Let $\Sigma$ be a simplicial complex on $[d+n+2] = \{1,2,\dots,d+n+2\}$ such that $2^{[d+n+2]} \setminus \Sigma$ can be partitioned into $n$ families $\F_1, \dots, \F_n$, where each $\F_i$ is pairwise intersecting. Then for any continuous map $f\colon |\Sigma| \to \R^d$, there are disjoint faces $\sigma, \tau \subset |\Sigma|$ such that $f(\sigma) \cap f(\tau) \ne \emptyset$.
\end{theorem}

This result indeed guarantees the existence of certain Radon partitions: The simplicial complex $\Sigma$ encodes which subsets of $[d+n+2]$ are permitted to appear in Radon partitions; for face-wise linear $f\colon |\Sigma| \to \R^d$ the image $f(\sigma)$ of a face $\sigma$ of $\Sigma$ is the convex hull of the images of its vertex set. In Section~\ref{sec:dolnikov} we will show that Dolnikov's theorem follows from Theorem~\ref{thm:sarkaria} via Gale duality. 

To prove our main results we apply Gale duality to the vertices of all polytopes in $\bigcup \F_i$ and points at their pairwise intersections. We thus need to establish a Radon-type result (for multiple Radon partitions that interact in a specific way) to establish hyperplane transversal results.
	
We denote the simplex on vertex set $[n+1]=\{1,\ldots, n+1\}$ by~$\Delta_n$. Our notation does not distinguish between $\Delta_n$ as an abstract simplicial complex and its geometric realization. Given a continuous map ${f\colon \Delta_n\rightarrow \R^d}$, we say that a tuple of faces $(\sigma^+,\sigma^-)$ of $\Delta_n$ is a \emph{Radon pair} for $f$ if $\sigma^+\cap \sigma^-=\emptyset$ and  $f(\sigma^+)\cap f(\sigma^-)\ne \emptyset$. The Gale dual of Theorem~\ref{thm:two-transversal} is the linear case of the following Radon-type result:

\begin{theorem}
\label{thm:two-radon}
	Let $m=2^s+t$, where $s\geq 1$ is an integer and $t\in \{-1,0, 1\}$. Let $d\geq 1$ be an integer and let $n \ge d+\frac32m+2$.
	Let $\F$ be a family of subsets of~$[n]$ with $\chi(\KG^4(\F)) \le m$. Then for any continuous map $f\colon \Delta_{n-1} \to \R^d$, there are two Radon pairs $(\sigma^+, \sigma^-)$ and $(\tau^+,\tau^-)$ for~$f$ such that none of the four intersections $\sigma^\pm \cap \tau^\pm$ contains a set from~$\F$.
\end{theorem} 

Given a generic set $X$ of $d+2$ points in~$\R^d$ there is a unique way to partition $X$ into $A$ and $B$ with $\conv A \cap \conv B \ne \emptyset$. Theorem~\ref{thm:two-radon} asserts that for $X \subset \R^d$ with $d+4$ points, there are two such partitions $A \sqcup B$ and $C \sqcup D$ that are maximally different, that is, the intersections $A \cap C$, $A \cap D$, $B\cap C$, and $B \cap D$ have at most $\frac{d+4}{4}$ points. This follows by letting $m=1$ and $\F$ the family of subsets of $[d+4]$ that have more than $\frac{d+4}{4}$ elements.

Theorem~\ref{thm:two-radon} and its relatives for more than two Radon pairs imply our main results on hyperplane transversals. To establish these results we use a standard translation to nonexistence results for certain equivariant maps. If Theorem~\ref{thm:two-radon} failed, then there would exist a continuous map $S^n \times S^n \to \R^{2d+2} \times \R^{3m}$ that avoids zero and which commutes with certain actions of the Dihedral group $D_4$. Here $S^n \times S^n$ parametrizes all possible Radon pairs $(\sigma^+, \sigma^-)$ and $(\tau^+,\tau^-)$. The $D_4$-action independently swaps $\sigma^+ \leftrightarrow \sigma^-$ and $\tau^+ \leftrightarrow \tau^-$ and interchanges $(\sigma^+, \sigma^-)$ with $(\tau^+,\tau^-)$. We will explain this translation to the equivariant setup in Section~\ref{sec:general}. The non-existence results for equivariant maps that we establish are new, and in fact extend earlier non-existence results that were derived in the context of hyperplane equipartitions. 

\section{Dolnikov's theorem via Gale duality}
\label{sec:dolnikov}

In this section we provide a new proof of Theorem~\ref{thm:dolnikov} which (essentially) exhibits this result as the Gale dual of the linear version of Sarkaria's theorem. For the sake of completeness, we first review the construction of the Gale transform as well as some of its essential properties which we use throughout the paper. For a more complete exposition, see e.g.~{\cite[Ch.~5.6]{Ma12}}.

\subsection{Gale Transform: Construction and Properties} Let $a_1,\ldots, a_n$ be a sequence of $n$ points in $\R^d$ whose affine hull is~$\R^d$. In particular, $n\geq d+1$. For each $j\in [n]$, one lets $a_j'=(a_j,1)$ be the vector in $\R^{d+1}$ obtained by appending $1$ to $a_j$, and denotes by $A$ the $(d+1)\times n$ matrix whose columns are the $a_j'$. As the $a_j$ affinely span $\R^d$, $\rank(A)=d+1$. Given any basis $v_1,\ldots, v_{n-d-1}$ for the nullspace $V=\Null(A)$, one lets $B$ be the $(n-d-1)\times n$ whose rows are the $v_i$. The Gale transform of $a_1,\ldots, a_n$ is defined to be the sequence $b_1,\ldots, b_n$ in $\R^{n-d-1}$ made of the columns of~$B$. It is straightforward to verify that (1) the $b_j$ linearly span $\R^{n-d-1}$ and 
that (2) $\sum_{j=1}^n b_j=0$. Moreover, one sees that any sequence $b_1,\ldots, b_n$ of points in $\R^{n-d-1}$ satisfying (1) and (2) is the Gale transform of some sequence $a_1,\ldots, a_n$ of points in $\R^d$ which affinely span ~$\R^d$.

For our purposes, the main property of the Gale transform that we use is the duality it provides between Radon pairs of point sets in $\R^d$ and partitions of point sets in $\R^{n-d-1}$ by linear hyperplanes. The precise formulation of the connection we need is stated as Proposition~\ref{prop:Gale} below. We shall say that a Radon pair $(A^+,A^-)$ of a finite point set $A\subset \R^d$ is ~\emph{minimal} if $\conv(A^+)\cap \conv(A')=\emptyset$ for any proper subset $A'\subset A^-$ and likewise $\conv(A')\cap \conv(A^-)=\emptyset$ for any proper subset $A'\subset A^+$.

\begin{proposition}
\label{prop:Gale} Let $d$ and $n$ be positive integers with $n\geq d+1$.
\begin{compactenum}[(i)]
\item Suppose that $(A^+,A^-)$ is a minimal Radon pair for a set $A=\{a_1,\ldots, a_n\}$ of $n$ points in $\R^d$ whose affine hull is $\R^d$. Let $I^+=\{j\in [n]\colon a_j\in A^+\}$ and $I^-=\{j\in [n]\colon a_j\in A^-\}$ be the corresponding index sets. Let $b_1,\ldots, b_n$ in $\R^{n-d-1}$ be the Gale transform of the sequence $a_1,\ldots, a_n$. Then there exists a linear hyperplane $H$ in $\R^{n-d-1}$ such that $B^+=\{b_j\colon j\in I^+\}\subset H^+$, $B^-=\{b_j\colon j\in I^+\}\subset H^-$, and all remaining $b_j$ lie on $H$.

\item Let $B=\{b_1,\ldots, b_n\}$ be a set of $n$ points in $\R^{n-d-1}$ such that $\sum_{j=1}^nb_j=0$ and which linearly span $\R^{n-d-1}$. Suppose that there is a linear hyperplane $H$ in $\R^{n-d-1}$ and non-empty subsets $B^+,B^-\subset B$ such that $B^+\subset H^+$, $B^-\subset H^-$, and all the remaining $b_j$ lie on $H$. Let $J^+=\{j\in [n]\colon b_j\in B^+\}$ and $J^-=\{j\in [n]\colon b_j\in B^-\}$ be the corresponding index sets. Then there is a sequence $a_1,\ldots, a_n$ of $n$ points in $\R^d$ whose Gale transform is $b_1,\ldots, b_n$ and subsets $I^+\subset J^+$ and $I^-\subset J^-$ such that $(\conv\{a_j\colon j\in I^+\}, \conv\{a_j\colon j\in I^-\})$ is a minimal Radon pair for $A=\{a_1,\ldots, a_n\}$, $\{b_j\colon j\in I^+\}\subset B^+$, and $\{b_j\colon j\in I^-\}\subset B^-$. 
\end{compactenum}
\end{proposition}

While Proposition~\ref{prop:Gale} can be established using standard methods (see, e.g.~~{\cite[Lemma.~5.6.2]{Ma12} for an essentially similar statement), for the sake of completeness we provide a proof.

\begin{proof}[Proof of Proposition~\ref{prop:Gale}]
For part (i), suppose that $(A^+,A^-)$ is a minimal Radon pair for $\{a_1,\ldots, a_n\}$ and let $I^+$ and $I^-$ be the corresponding index sets of $[n]$. In particular, there exist scalars $\lambda_1,\ldots, \lambda_n \geq 0$ such that $\lambda_j>0$ for all $j\in I^+\cup I^-$, $\sum_{j\in I^+} \lambda_j = \sum_{j\in I^-}\lambda_j=1$, and $\sum_{j\in I^+}\lambda_j a_j=\sum_{j\in I^-}\lambda_j a_j$. Defining $t=(t_1,\ldots, t_n)$ by $t_j=\lambda_j$ for $j\in I^+$, $t_j=-\lambda_j$ for $j\in I^-$, and $t_j=0$ otherwise, we have $\sum_jt_ja_j=0$ and $\sum_j t_j=0$. Using the notation of the construction of the Gale transform, we see that $t$ lies in $V=\Null(A)$ so $t=\sum_{i=1}^{n-d-1}\alpha_iv_i$ with $\alpha=(\alpha_1,\ldots, \alpha_{n-d-1})\in \R^{n-d-1}$. Considering the linear hyperplane $H=\langle \alpha \rangle^\perp$, it is easily verified that $t_j= \langle b_j,\alpha \rangle$ for all $j\in [n]$. Thus $b_j\in H^+$ for all $j\in I^+$, $b_j\in H^-$ for all $j\in I^-$, and $b_j\in H$ if $j\notin I^+\cup I^-$. 

To prove (ii), suppose that there exist non-empty subsets $B^+$ and $B^-$ of $\{b_1,\ldots, b_n\}$ and a linear hyperplane $H$ such that $B^+\subset H^+$, $B^-\subset H^-$, and all other points of $b_j$ lie on $H$. Let $J^+$ and $J^-$ be the corresponding index sets of $[n]$. We have that $J^+\cap J^-=\emptyset$ because the half-spaces are open. We have $H=\langle \alpha \rangle^\perp$ in $\R^{n-d-1}$ for some $\alpha=(\alpha_1,\ldots, \alpha_{n-d-1})\in \R^{n-d-1}\setminus\{0\}$. Define $t=(t_1,\ldots, t_n)$ by $t_j=\langle b_j,\alpha\rangle$ for all $j\in [n]$. Thus $t_j>0$ for all $j\in J^+$, $t_j<0$ for all $j\in J^-$, and $t_j=0$ for all $j\notin J^+\cup J^-$. 
 Letting $v_1,\ldots, v_{n-d-1}$ be the rows of the matrix $B$ whose columns are the $b_j$, one verifies that
$t=\sum_{i=1}^{n-d-1}\alpha_i v_i$, and it is again straightforward to show that $\sum_{j=1}^n t_j a_j=0$ and $\sum_{j=1}^n t_j=0$. Now define $\lambda_j=\frac{t_j}{\sum_{i\in J^+}t_i}$ if $j\in J^+$ and $\lambda_j=\frac{-t_j}{\sum_{i \in J^-}t_i}$ if $j\in J^-$. Thus $\lambda_j>0$ for all $j\in J^+\cup J^-$, $\sum_{j\in J^+}\lambda_j=\sum_{j\in J^-}\lambda_j=1$, and $\sum_{j\in J^+}\lambda_ja_j=\sum_{j\in J^-}\lambda_ja_j$. To complete the proof, by removing points if necessary we may choose $I^+\subset J^+$ and $I^-\subset J^-$ so that $A^+=\{a_i\colon i\in I^+\}$ and $A^-=\{a_i\colon i\in I^+\}$ are minimal. 
\end{proof}

\subsection{A Proof of Dolnikov's Theorem}

We now derive Theorem~\ref{thm:dolnikov} as a consequence of the linear version of Theorem~\ref{thm:sarkaria}; the reverse implication is given in Remark~\ref{rem:Soberon} following Lemma~\ref{lem:equivalence}. As we have already seen, Theorem~\ref{thm:hamsandwich} is a special case of Theorem~\ref{thm:dolnikov}. In order to use Gale duality, it will be convenient to prove the following slight reformulation of Theorem~\ref{thm:dolnikov}. Recall that $\KG(\F)=\KG^2(\F)$. 

\begin{proposition}
\label{prop:dolnikov-finite}
    Let $Y \subset \R^d$ be a finite set and let $\F$ be a family of subsets of~$Y$ such that ${\chi(\mathrm{KG}(\F)) \le d}$. Then there exists an affine hyperplane $H$ in~$\R^d$ which intersects all the convex hulls of all sets in~$\F$.
\end{proposition}

To see that Proposition~\ref{prop:dolnikov-finite} implies 
Theorem~\ref{thm:dolnikov}, assume Proposition~\ref{prop:dolnikov-finite} and let $\F_1, \dots, \F_d$ be as in the statement of Theorem~\ref{thm:dolnikov}. For each $\F_i$, let $Y_i\subset \R^d$ be the set that consists of (i) the vertices from each polytope $P$ of ~$\F_i$ together with (ii) a point from the intersection of any pair of distinct polytopes $P$ and $Q$ from ~$\F_i$. Now let $Y=\cup_i Y_i$ and let $\F'=\cup_i \F'_i$, where for each $i$ we define $\F'_i = \{P \cap Y_i \ : \ P \in \F_i\}$. Each $\F'_i$ is pairwise intersecting by construction, so ${\chi(\mathrm{KG}(\F')) \le d}$. By Proposition~\ref{prop:dolnikov-finite}, there is therefore an affine hyperplane $H \subset \R^d$ that intersects $\conv (P\cap Y_i)=P$ of each $P$ in $\F_i$ and every $\F_i$. We now prove Proposition~\ref{prop:dolnikov-finite}. 

\begin{proof}[Proof of Proposition~\ref{prop:dolnikov-finite} via Theorem~\ref{thm:sarkaria}]
 Identify $\R^d$ with the hyperplane in $\R^{d+1}$ with last coordinate equal to~1 and let $Y\subset \R^d$ and $\F$ be as in the statement of Proposition~\ref{prop:dolnikov-finite}. By adding points to the set $Y\subset \R^d$, if necessary, we obtain a set $Y_1$ which linearly spans~$\R^{d+1}$. Note that no set in $\F$ will contain any of the added points. Letting $Y_2 = Y_1 \cup \{-\sum_{y \in Y_1} y\}$ now gives a set which linearly spans ~$\R^{d+1}$ and has barycenter the origin. Thus $Y_2$ is the Gale transform of some finite set $X\subset\R^{n-d-2}$ which affinely spans $\R^{n-d-2}$. For concreteness, say that $Y_2=\{b_1,\ldots, b_n\}$ and that $X=\{a_1,\ldots, a_n\}$. 
 
 To avoid a slight abuse of notation, let $\F'$ be the family of subsets of $[n]=\{1,\ldots, n\}$ corresponding to the family $\F$ of subsets of the original set ~$Y$. We let $\Sigma(\F')$ denote the simplicial complex on $[n]$ whose minimal non-faces are the sets of ~$\F'$, that is, $$\Sigma(\F')=\{\sigma \subset [n] \ : \ \tau \notin \F'\,\, \text{for all}\,\, \tau\subseteq \sigma\}.$$
\noindent The set family $2^{[n]} \setminus \Sigma(\F')$ is obtained from $\F'$ by including any supersets of elements of~$\F'$, so in particular $\chi(\KG(2^{[n]} \setminus \Sigma(\F')) = \chi(\KG(\F)) \leq d$. Thus $2^{[n]} \setminus \Sigma(\F')$ can be partitioned into $\F'_1, \dots, \F'_d$, where each $\F'_i$ is pairwise intersecting. 

The linear version of Theorem~\ref{thm:sarkaria} gives two disjoint faces $\sigma^+, \sigma^-\in \Sigma(\F')$ such that $$\conv\{a_i \ : i \in \sigma^+\} \cap \conv \{a_i \ : i \in \sigma^-\} \ne \emptyset.$$ Without loss of generality, we may assume that $\{a_i \colon i \in \sigma^+\}$ and $\{a_i \colon i \in \sigma^-\}$ form a minimal Radon pair, so that by Proposition~\ref{prop:Gale} there is a linear hyperplane~$H$ in $\R^{d+1}$ such that $\{b_i \ :  i \in \sigma^+\}$  lies in $H^+$, $\{b_i \ : i \in \sigma^-\}$ lies in $H^-$, and $H$ contains every other point of ~$Y_2$. By the definition of $\Sigma(\F')$, neither $\sigma^+$ nor $\sigma^-$ can contain any set from~$\F'$. Thus neither $H^+$ nor $H^-$ can contain any set of the form $\{b_i \ : \ i \in F'\}$ with $F' \in \F'$, which is to say that neither open half space contains any set from $\F$. Therefore, $H$ intersects the convex hull of every set of $\F$. On the other hand, $\F$ is a family of subsets of the original set $Y\subset \R^d$. As the convex hull of each set of $\F$ lies in $\R^d$ and the intersection of $H$ and $\R^d$ is an affine hyperplane in~$\R^d$, the proof is complete.
\end{proof}

\section{Multiple hyperplane transversals via Radon-type results}

As for a single hyperplane, equipartition results by multiple hyperplanes are special cases of hyperplane transversal generalizations, where now the transversal consists of a union of several hyperplanes. For linear hyperplanes, by Gale duality these transversal results in turn have equivalent formulations as Radon-type results for several Radon pairs. We shall say that a Radon pair $(\sigma^+, \sigma^-)$ for $f\colon \Delta_n \to \R^d$ is \emph{minimal} if both $f(\sigma^+) \cap f(\tau) = \emptyset$ for any proper subface $\tau \subset \sigma^-$ and $f(\sigma^-) \cap f(\tau) = \emptyset$ for any proper subface $\tau \subset \sigma^+$. 

\begin{lemma}
\label{lem:equivalence} 
    Let $c \ge 1$, $d \ge 1$, $k\ge 1$, and $m\ge 1$ be integers. The following are equivalent:
    \begin{compactenum} 
        \item\label{it:radon} If $n\ge d+c+2$, then for any linear map $f\colon \Delta_{n-1}\to \mathbb{R}^d$ and for any family $\F$ of subsets of $[n]$ with $\chi(\KG^{2^k}(\F)) \le m$, there are $k$ minimal Radon pairs $(\sigma_1^+,\sigma_1^-),\ldots, (\sigma_k^+,\sigma_k^-)$ for~$f$ such that none of the 
        intersections of the form $\sigma_1^\pm \cap \dots \cap \sigma_k^\pm$ contains a set from~$\F$.
        \item\label{it:linear-piercing} For any finite family $\F$ of polytopes in $\mathbb{R}^{c+1}$ with $\chi(\KG^{2^k}(\F)) \le m$, there are $k$ linear hyperplanes which pierce $\F$. 
    \end{compactenum} 
\end{lemma}

\begin{remark}
\label{rem:Soberon}
Using Lemma~\ref{lem:equivalence}, let us now show how the linear case of Sarkaria's theorem follows from Dolnikov's theorem. Assuming Theorem~\ref{thm:dolnikov}, let $\F$ be a finite family of polytopes in $\mathbb{R}^{d+1}$ with $\chi(\KG(\F)) \le d$. Adding the singleton containing the origin to $\F$ if necessary gives a finite family $\F'$ of polytopes in $\R^{d+1}$ with $\chi(\KG(\F'))\leq d+1$.  Theorem~\ref{thm:dolnikov} guarantees a hyperplane  which pierces each polytope in $\F'$, and this hyperplane necessarily passes through the origin while intersecting every polytope in $\F$. This establishes (\ref{it:linear-piercing}) and therefore (\ref{it:radon2}) of Lemma~\ref{lem:equivalence} when $k=1$ and $c=m$. On the other hand, if $\F$ is a family of subsets of $[d+m+2]$ with $\chi(\KG(\F))\leq m$, let $\Sigma=\{\sigma\subset [d+m+2]\colon \tau\notin \F\,\text{for all}\, \tau\subseteq \sigma\}$ be the simplicial complex of non-faces. Any given linear map $f\colon \Sigma\rightarrow \mathbb{R}^d$ extends linearly to a linear map from $\Delta_n$ to $\mathbb{R}^d$, so Lemma~\ref{lem:equivalence}(\ref{it:radon2}) gives the linear case of Theorem~\ref{thm:sarkaria}. \end{remark}

\begin{proof}[Proof of Lemma~\ref{lem:equivalence}]
Assuming~(\ref{it:radon}), let $\F$ be a finite family of polytopes in $\mathbb{R}^{c+1}$ with $\chi(\KG^{2^k}(\F)) \le m$. As before, let $Y$ be the set that contains all vertices of all these polytopes as well as a point in the intersection of any two polytopes $P, Q \in \F$ with $P \cap Q \ne \emptyset$. Again as before, we may artificially add points to $Y$ if necessary so that $Y = \{y_1, \dots, y_n\}$ affinely spans $\R^{c+1}$ and $\sum_{y \in Y} y= 0$. Thus $Y$ is the Gale transform of some $X = \{x_1, \dots, x_n\} \subset \R^{n-c-2}$. Let $\varphi\colon Y \to [n]$ be the bijection $\varphi(y_i) = i$. The corresponding points $x_1, \dots, x_n$ in $\R^{n-c-2}$ thus determine a linear map $f\colon \Delta_{n-1} \to \R^d$ with $d = n-c-2$. 

 Let $\F' = \{\varphi(P \cap Y) \ : \ P \in \F\}$. By construction of the set $Y$, it follows that $\KG^{2^k}(\F') = \KG^{2^k}(\F)$ and so $\chi(\KG^{2^k}(\F')) \le m$. By~(\ref{it:radon}), there are $k$ minimal Radon pairs $(\sigma_1^+, \sigma_1^{-}), \dots, (\sigma_k^+, \sigma^{-}_k)$ for~$f$ such that none of the $2^k$ intersections of the form $\sigma_1^\pm \cap \dots \cap \sigma_k^\pm$ contains any set in~$\F'$. Each pair $(\sigma_j^+, \sigma_j^{-})$ determines intersecting convex hulls $f(\sigma_j^+) \cap f(\sigma_j^-) \ne \emptyset$ of points in~$X$. By Proposition~\ref{prop:Gale}, for each $j\in[k]$ there is therefore a linear hyperplane $H_j \subset \R^{c+1}$ such that $\{y_i  \ : \ i \in \sigma_j^+\}$ lies in~$H_j^+$, $\{y_i  \ : \ i \in \sigma_j^-\}$ lies in $H_j^-$, and $H_j$ passes through every other point of~$Y$. 
    
    Now let $P \in \F$ be arbitrary. Letting $A = P \cap Y$ we have that $\varphi(A) \in \F'$ and $\conv A = P$. By (\ref{it:radon2}), there is some $j \in [k]$ such that $\varphi(A)$ is not fully contained in either $\sigma_j^+$ or~$\sigma_j^-$. Thus $A=\{y_i \ : \ i \in \varphi(A)\}$ is neither completely to the open positive side nor to the negative side of~$H_j$, and therefore $H_j$ intersects $P = \conv A$.
    
   The proof that (\ref{it:linear-piercing}) implies (\ref{it:radon2}) essentially follows the proof of ~\cite[Theorem 5]{So15} via Gale duality and the ham sandwich theorem. Given a linear map $f\colon \Delta_{n-1}\to \mathbb{R}^d$ and a family $\F$ of subsets of $[n]$ with ${\chi(\KG^{2^k}(\F)) \le m}$ as in~(\ref{it:radon}), let $x_i = f(i)$ and $X= \{x_1, \dots, x_n\}$. We may  assume that $n = d+c+2$, since otherwise one may simply restrict to a subset of the vertices. By possibly decreasing $d$, we may assume that $X$ affinely spans~$\R^d$. Denote the Gale transform of $X$ by $Y = \{y_1, \dots, y_n\} \subset \R^{n-d-1} = \R^{c+1}$ and let $\F'$ be the family consisting of all polytopes of the form $P_A=\conv \{y_i \ : \ i \in A\}$ for each $A \in \F$. If $P_A\cap P_B=\emptyset$, 
    then $A \cap B = \emptyset$. Thus $\KG^{2^k}(\F')$ is contained in~$\KG^{2^k}(\F)$, and so in particular $\chi(\KG^{2^k}(\F')) \le m$.
    
    Assuming~(\ref{it:linear-piercing}), there are $k$ linear hyperplanes $H_1, \dots, H_k \subset \R^{c+1}$ such that every polytope in $\F'$ is intersected by $\cup_i H_i$. As before, let $\varphi\colon Y \to [n]$ be the bijection $\varphi(y_i) = i$. By Proposition~\ref{prop:Gale}, for each $j\in [k]$ there is minimal Radon pair $(\sigma_j^+,\sigma_j^-)$ for $f$ with $\sigma_j^+\subset  \varphi(H_j^+ \cap Y)$ and $\sigma_j^- \subset \varphi(H_j^- \cap Y)$ for each $j \in [k]$. Now let $A \in \F$ be arbitrary. We have that $\conv \{y_i \ : \ i \in A\}$ is in $\F'$ and so must be pierced by some~$H_j$. Thus $\{y_i\colon i\in A\}$ does not lie in either open half-space determined by $H_j$, so $A$ cannot be contained in either ~$\sigma_j^+$ or ~$\sigma_j^-$. In particular, $A$ does not lie in any of the intersections $\sigma_1^\pm\cap\cdots\cap\sigma_k^\pm$.
\end{proof}

The following is our central result on the connection between Radon-type theorems, hyperplane transversality, and mass equipartitions.

\begin{theorem}
\label{thm:implications}
    Let $c \ge 1$, $d \ge 1$, $k\ge 1$, and $m\ge 1$ be integers. Of the statements below, (\ref{it:radon2}) implies~(\ref{it:affine-piercing}), which in turn implies (\ref{it:equipart}).
    \begin{compactenum} 
        \item\label{it:radon2} If $n\ge d+c+2$, then for any linear map $f\colon \Delta_{n-1}\to \mathbb{R}^d$ and for any family $\F$ of subsets of $[n]$ with $\chi(\KG^{2^k}(\F)) \le m$, there are $k$ minimal Radon pairs $(\sigma_1^+,\sigma_1^-),\ldots, (\sigma_k^+,\sigma_k^-)$ for~$f$ such that none of the
        intersections of the form $\sigma_1^\pm \cap \dots \cap \sigma_k^\pm$ contains a set from~$\F$.
        \item\label{it:affine-piercing} For any finite family $\F$ of polytopes in $\mathbb{R}^{c}$ with $\chi(\KG^{2^k}(\F)) \le m$, there are $k$ affine hyperplanes which pierce $\F$.
        \item\label{it:equipart} For finite point sets $X_1, \dots, X_m \subset \R^c$, there are $k$ affine hyperplanes $H_1, \dots, H_k$ such that each of the 
        orthants $H_1^\pm \cap \dots \cap H_k^\pm$ contains no more than $\frac{1}{2^k}|X_i|$ points of each~$X_i$, i.e., $\Delta(m,k)\leq c$.
    \end{compactenum} 
\end{theorem}

\begin{remark}
\label{rem:optimal} 
Observe that any $c\geq1$ for which any of (\ref{it:radon2}) through (\ref{it:equipart}) of Theorem~\ref{thm:implications} holds must satisfy $c\ge \lceil\frac{2^k-1}{k}m\rceil$. In particular, the dimension $d$ of Theorem~\ref{thm:two-transversal} is tight with respect to~$m$, as is the $n$ of Theorem~\ref{thm:two-radon}. 
\end{remark}

\begin{proof}[Proof of Theorem~\ref{thm:implications}]
    We first show that (\ref{it:radon2}) implies~(\ref{it:affine-piercing}). By Lemma~\ref{lem:equivalence}, we have that (\ref{it:radon2}) is equivalent to the statement: For any finite family $\F$ of polytopes in $\mathbb{R}^{c+1}$ with $\chi(\KG^{2^k}(\F)) \le m$ there are $k$ linear hyperplanes $H_1, \dots, H_k \subset \mathbb{R}^{c+1}$ such that every polytope in $\F$ is intersected by~$\cup_{i=1}^k H_i$. If instead we are given a finite family $\F$ of polytopes in $\mathbb{R}^{c}$ with $\chi(\KG^{2^k}(\F)) \le m$, we place $\R^c$ in~$\R^{c+1}$ as the hyperplane at height one. Then by the above there are $k$ linear hyperplanes $H_1, \dots, H_k \subset \mathbb{R}^{c+1}$ such that every polytope in $\F$ is intersected by~$\cup_{i=1}^k H_i$. These hyperplanes $H_i$ correspond to affine hyperplanes in~$\R^c$, thus proving~(\ref{it:affine-piercing}).
    
    To show that (\ref{it:affine-piercing}) implies (\ref{it:equipart}) (and in fact is a special case), let finite sets $X_1, \dots, X_m \subset \R^c$ be given. For each $i\in [m]$, consider the corresponding family $$\F_i = \left\{\conv A \ : \ A \subset X_i, \ |A| > \frac{1}{2^k}|X_i|\right\}.$$ Letting $\F = \bigcup_i \F_i$ we see that $\chi(\KG^{2^k}(\F)) \le m$, and so by~(\ref{it:affine-piercing}) there are $k$ affine hyperplanes $H_1, \dots, H_k \subset \mathbb{R}^{c}$ such that every polytope in $\F$ is intersected by~$\cup_{i=1}^k H_i$. If $A \subset X_j$ is any set with $|A| > \frac{1}{2^k}|X_j|$, then by construction its convex hull $\conv A$ is intersected by some hyperplane~$H_i$, so $A$ cannot lie in any open orthant determined by the hyperplanes.
\end{proof}

\section{A general bound for multiple Radon pairs}
\label{sec:general}

Over the next three sections we explain how to obtain topological Radon-type theorems from non-existence results for equivariant maps. These follow from the well-established configuration space/test map scheme and a generalization of Borsuk--Ulam type results arising from the Gr\"unbaum--Ramos--Hadwiger mass partition problem.

In this section we will prove Theorem~\ref{thm:upper-transversal}, which by Theorem~\ref{thm:implications} is a consequence of the following result about Radon pairs for continuous maps. As before, let $m=2^p+q$ where $p\ge 0$ and $0\le q<2^p$, and let $U(m,k)=2^{k+p-1}+q$ be the upper bound of Theorem~\ref{thm:hyperplane-bound} for the Gr\"unbaum--Ramos--Hadwiger problem. As $m=2^{p+1}-1$ when $k=2$ and $q=2^{p-1}-1$, the $t=-1$ case of Theorem~\ref{thm:two-radon} is an immediate consequence.

\begin{theorem}
\label{thm:upper}
    Let $m\ge 1$, $k\ge 1$, and $d\ge 1$. Suppose that $n\ge d+U(m,k)+1$. If $\F$ is a family of subsets of $[n+1]$ with $\chi(\KG^{2^k}(\F)) \le m$, then for any continuous map $f\colon\Delta_n\rightarrow \R^d$ there are $k$ Radon pairs $(\sigma_1^+, \sigma_1^{-}), \dots, (\sigma_k^+, \sigma^{-}_k)$ for~$f$ such that none of the intersections of the form $\sigma_1^\pm \cap \dots \cap \sigma_k^\pm$ contains a set from~$\F$. 
\end{theorem}

\subsection{Configuration Space} Let $\Delta_n^{\ast 2} =\{\sigma \ast \tau\mid \sigma,\tau\subseteq \Delta_n\}$ be the join of the $n$-simplex with itself. The deleted join $\Sigma_n:=(\Delta_n)^{\ast 2}_\Delta=\{\sigma^+\ast\sigma^-\mid \sigma^+\cap\sigma^-=\emptyset\}$ is the sub-complex of $\Delta_n^{\ast 2}$ consisting of all  formal convex sums $\lambda x^+ + (1-\lambda) x^-$, where $x^+\in \sigma^+$ and $x^-\in \sigma^-$, and for which the $\sigma^\pm$ are disjoint. This complex is realized as the boundary of the $n$-dimensional cross--polytope, that is, the unit sphere in the $\ell_1$-metric on $\R^{n+1}$, and is equipped with a $\mathbb{Z}_2$-action obtained by interchanging $\lambda x^+$ with $(1-\lambda)x^-$. Under the identification of $\Sigma_n$ with the $n$-sphere $S^n$, this action corresponds to the usual antipodal action on the sphere.  

For any $k\ge 2$, let $$\Sigma(n,k)=\Sigma_n^{\times k}=\{\sigma=(\sigma_1^+\ast \sigma_1^-,\cdots, \sigma_k^+\ast \sigma_k^-)\mid \sigma_i^+\cap \sigma_i^-=\emptyset\,\,\text{for all}\,\, 1\leq i \leq k\}$$ be the $k$-fold product of $\Sigma_n$. This product comes equipped with a canonical action of the hyperoctahedral group $\mathfrak{S}_k^\pm=\mathbb{Z}_2^k\rtimes \mathfrak{S}_k$, whereby $\mathbb{Z}_2^k$ acts by interchanging each pair $\{\sigma_i^+,\sigma_i^-\}$ of disjoint faces for any $1\le i \leq k$ and the symmetric group $\mathfrak{S}_k$ acts by permuting the $\sigma_i^+\ast \sigma_i^-$. This action is by isometries if one takes the product metric on $\Sigma(n,k)$ with the $\ell_1$-metric on each $\Sigma_n$ factor. Extending the identification of $\Sigma_n$ with $S^n$, one has a natural correspondence between $\Sigma(n,k)$ and the product $Y(n,k):=(S^n)^k$ of $n$-spheres, the later being equipped with the analogous $\mathfrak{S}_k^\pm$-action whereby $\mathbb{Z}_2^k$ acts antipodally on each coordinate sphere and $\mathfrak{S}_k$ acts by permuting each $S^n$ factor. Explicitly, let $(y_1,\ldots, y_k)\in (S^n)^k$, let $(g_1,\ldots, g_k)\in \mathbb{Z}_2^k$, and let $\alpha\in \mathfrak{S}_k$. Then 
 $$((g_1,\ldots, g_k)\rtimes \alpha)\cdot (y_1,\ldots, y_k) =((-1)^{g_1} y_{\alpha^{-1}(1)},\ldots, (-1)^{g_k} y_{\alpha^{-1}(k)}).$$
 
Observe that the action of $\mathbb{Z}_2^k$ on $\Sigma_{n,k}$ (and $Y(n,k)$) is free, and in fact it is easily observed that $g\cdot \sigma \cap \sigma=\emptyset$ for all non-identity $g=(g_1,\ldots, g_k)\in \mathbb{Z}_2^k$ and all $\sigma=(\sigma_1^+\ast \sigma_1^-,\ldots, \sigma_k^+\ast\sigma_k^-)\in \Sigma(n,k)$. On the other hand, the full action of $\mathfrak{S}_k^\pm$ is \textit{not free}, with non-free part $\Sigma(n,k)'$ consisting of all $([\lambda_1x_1^++(1-\lambda_1)x_1^-],\ldots, [\lambda_kx_k^++(1-\lambda_k)x_k^-])\in \Sigma(n,k)$ such that $\lambda_ix_i^++(1-\lambda_i)x_i^-=\pm (\lambda_jx_j^++(1-\lambda_j)x_j^-)$ for some $i\neq j$. Under the identification of $\Sigma(n,k)$ with $Y(n,k)$, this corresponds to all $(y_1,\ldots, y_k)\in Y(n,k)$ with $y_i=\pm y_j$ for distinct $i$ and $j$, which is the non-free part $Y'(n,k)$ of $Y(n,k)$ under the $\mathfrak{S}_k^\pm$-action.

\subsection{Test Space} Let $\R[\mathbb{Z}_2^k]=\left\{\sum_{g\in \mathbb{Z}_2^k} r_g g\mid r_g\in \R\,\, \text{for all}\,\, g\in \mathbb{Z}_2^k\right\}$ denote the regular representation of $\mathbb{Z}_2^k$. Considering the action of the symmetric group on coordinates, $\R[\mathbb{Z}_2^k]$ becomes a $\mathfrak{S}_k^\pm$-module, as does $U_k=\left\{\sum_{g\in \mathbb{Z}_2^k} r_g g\mid \sum_{g\in \mathbb{Z}_2^k} r_g=0\right\}$, the orthogonal complement of the trivial $\mathbb{Z}_2^k$-representation.  We consider the direct sum $U_k^{\oplus m}$ as a $\mathfrak{S}_k^\pm$-module under the diagonal action.

In addition to $U_k^{\oplus m}$, we shall also consider the $\mathfrak{S}_k^\pm$-module $V_k:=(\mathbb{R}^{d+1})^k$, where again $\mathfrak{S}_k$ acts by permuting each $\mathbb{R}^{d+1}$ factor and $\mathbb{Z}_2^k$ acts independently and antipodally on each coordinate $\mathbb{R}^{d+1}$. 

\subsection{Test Map}

Let $n=d+c+1$ for some constant $c$. To obtain $k$ Radon pairs for a given continuous map $f:\Delta_n\rightarrow \R^d$, for each $\lambda x=([\lambda_1x_1^++(1-\lambda_1)x_1^-],\ldots,[\lambda_kx_k^++(1-\lambda_k)x_k^-])\in \Sigma(n,k)$ we define $R\colon\Sigma(n,k)\rightarrow V_k$ by $$R(\lambda x)=(\lambda_1-1/2,\lambda_1f(x_1^+)-(1-\lambda_1)f(x_1^-),\ldots,\lambda_k-1/2, \lambda_k f(x_k^+)-(1-\lambda_k)f(x_k^-)).$$ It is easily observed that $R$ is equivariant with respect to the $\mathfrak{S}_k^\pm$-actions on $\Sigma(n,k)$ and $V_k$. 

Now let $\F$ be a family of subsets of the vertices of $\Delta_n$ such that~$\chi(\KG^{2^k}(\F)) \le m$. Fix a partition $\F_1 \cup \dots \cup \F_m$ of $\F$ with $\chi(\KG^{2^k}(\F_j)) \le 1$ for each $j \in [m]$. We may assume that $\F_j$ is closed under taking supersets. 
For each $\sigma=(\sigma_1^+\ast \sigma_1^-,\ldots, \sigma_k^+\ast\sigma_k^-)\in \Sigma(n,k)$, write $\sigma_i^0$ for $\sigma_i^+$ and write $\sigma_i^1$ for~$\sigma_i^-$. 
 For each $j \in [m]$ and each $g=(g_1,\ldots, g_k)\in \mathbb{Z}_2^k$, let $$K^j_{g}=\left\{\sigma=(\sigma_1^+\ast \sigma_1^-,\ldots, \sigma_k^+\ast \sigma_k^-)\in \Sigma(n,k)\mid \sigma_1^{g_1}\cap\cdots \cap \sigma_k^{g_k} \notin \F_j \right\}$$ and let $K^j=\cap_{g\in \mathbb{Z}_2^k}K^j_{g}$. Since $\F_j$ is closed under taking supersets, each $K^j_g$ and thus $K^j$ is a simplicial complex.

For each $\lambda x\in \Sigma(n,k)$ and each $g\in \mathbb{Z}_2^k$, consider the distance $d(\lambda x,K^j_g)$ from $\lambda x$ to the subcomplex $K^j_g$. Letting $$a_j(\lambda x)=\frac{1}{2^k}\sum_{g\in \mathbb{Z}_2^k} d(\lambda x,K^j_g)$$ be the average of these distances, define $D_j\colon\Sigma(n,k)\rightarrow \R[\mathbb{Z}_2^k]$ by $$D_j(\lambda x)=\sum_g (d(\lambda x,K^j_g)-a_j(\lambda x))g.$$ The image of $D_j$ thus lies in $U_k$, and so one has a mapping $D=(D_1,\ldots, D_m)\colon \Sigma(n,k)\rightarrow U_k^{\oplus m}$. Clearly, $D(\lambda x)=0$ if and only if for each fixed $g \in \mathbb{Z}_2^k$ we have that $\lambda x$ is equidistant from each~$K^j_g$. On the other hand, the sets $\sigma_1^{g_1}\cap\cdots \cap \sigma_k^{g_k}$ are pairwise disjoint. If each of these were to lie in some $\F_j$, then these intersections would all be distinct because no $\F_j$ contains the empty set. However, none of the $\F_j$ contains $2^k$ pairwise disjoint elements, and so for each $1\leq j\leq m$ and any $\lambda x\in \Sigma(n,k)$ one must have $d(\lambda x,K_{g(j)}^j)=0$ for some $g(j)\in \mathbb{Z}_2^k$. Any zero of $D$ must therefore lie in $K^j$ for all $1\leq j \leq m$, and conversely. As $g \sigma \cap \sigma=\emptyset$ unless $g$ is the identity and the $\mathbb{Z}_2^k$-action is by isometries, it is a straightforward verification that the mapping $D$ is $\mathfrak{S}_k^\pm$-equivariant with respect to the actions considered.

Letting $\iota\colon Y(n,k)\rightarrow \Sigma(n,k)$ be the natural ($\mathfrak{S}_k^\pm$-equivariant) identification of $Y(n,k)$ and $\Sigma(n,k)$, let \begin{equation}\label{eqn:F} F = (D\oplus R)\circ \iota\colon Y(n,k) \rightarrow U_k^{\oplus m}\oplus V_k.\end{equation} By construction, any zero of this map corresponds to the existence of $k$ Radon pairs that are not contained in~$\F$. Thus Theorem~\ref{thm:upper} follows if any $\mathfrak{S}_k^\pm$-map $F$ can be guaranteed a zero when $c\geq U(m,k)$. This will be established by appealing to a relevant theorem of Borsuk--Ulam type.

\subsection{Proof of Theorem~\ref{thm:upper}}

For Theorem~\ref{thm:upper}, we only need to consider the representation $U_k^{\oplus m}\oplus V_k$ when viewed as a $\mathbb{Z}_2^k$-module and apply a corresponding Borsuk--Ulam type theorem for $\mathbb{Z}_2^k$-equivariant maps. We recall that any $\mathbb{Z}_2^k$-representation is isomorphic to the direct sum of one-dimensional representations and that the latter are indexed by the group itself. Namely, for each $h=(h_1,\ldots, h_k)\in \mathbb{Z}_k$, the corresponding 1-dimensional representation $U_h$ given by $\chi_h(g)=(-1)^{h_1g_1+\cdots h_kg_k}$ for each $g\in \mathbb{Z}_2^k$. For example, for $U_k$ and $V_k$ as above one  has $U_k\cong \oplus_{h\in \mathbb{Z}_2^k-\{0\}}U_h$ and $V_k\cong \oplus_{i=1}^{k} U_{ \mathbf{e}_i}^{\oplus (d+1)}$, where $\mathbf{e}_i$ is the standard basis vector in $\mathbb{Z}_2^k$. While we shall find it convenient to use the following Borsuk--Ulam type result of ~\cite[ Proposition 6.1]{Si19} based on Stiefel--Whitney classes and the cohomology of real projective space, we note that similar results have been obtained via ideal-valued cohomological index theory~\cite{FH98} and used extensively in the field, such as in particular in ~\cite{MLVZ06} in establishing the upper bound of Theorem~\ref{thm:hyperplane-bound}. 

\begin{proposition}
\label{prop:cohomology} Let $h_1=(h_{1,1},\ldots, h_{1,k}),\ldots, h_{nk}=(h_{nk,1},\ldots, h_{nk,k})\in \mathbb{Z}_2^k$ and let $U=\oplus_{i=1}^{nk} U_{h_i}$. Let $P_U$ be the  polynomial in
$\mathbb{Z}_2[u_1,\ldots, u_k]/(u_1^{n+1},\ldots, u_k^{n+1})$
given by $$P_U(u_1,\ldots, u_k)=\prod_{i=1}^{nk} (h_{i,1}u_1+\cdots +h_{i,k}u_k).$$ If $P_U(u_1,\ldots, u_k) = u_1^n\cdots u_k^n$, then any $\mathbb{Z}_2^k$-equivariant continuous map $f\colon Y(n,k) \rightarrow U$ has a zero. 
\end{proposition}

\begin{proof}[Proof of Theorem~\ref{thm:upper}]

 Define $U=U_k^{\oplus m}\oplus V_k\oplus V'$, where for $m=2^p+q$ as above we set $V'=\oplus_{i=2}^k U_{\mathbf{e}_i}^{\alpha_i}$ with $\alpha_i=U(m,k)-2^{p+k-i}-q2^{i-1}$. Thus $\dim U=kn$, where $n=U(m,k)+d+1$. The corresponding polynomial $P_U$ can be expressed as the product $P_U=P_{U_k}^m\cdot P_{V_k}\cdot P_{V'}$, where $P_{U_k}=\prod_{(h_1,\ldots, h_k)\in \mathbb{Z}_2^k\setminus\{0\}}(h_1u_1+\cdots +h_ku_k)$, $P_{V_k}=x_1^{d+1}\cdots x_k^{d+1}$, and $P_{V'}=x_2^{\alpha_2}\cdots x_k^{\alpha_k}$. As shown in ~\cite{Wi83}, $P_{U_k}$ is the Dickson polynomial $P_{U_k}=\sum_{\sigma\in \mathfrak{S}_k} u_{\sigma(1)}^{2^{k-1}}u_{\sigma(2)}^{2^{k-2}}\cdots u_{\sigma(k)}^1$. Thus $$P_U=\sum_{\sigma\in \mathfrak{S}_k} u_{\sigma(1)}^{2^{k+p-1}}u_{\sigma(2)}^{2^{k+p-2}}\cdots u_{\sigma(k)}^{2^p}\cdot \left(\sum_{\tau \in \mathfrak{S}_k} u_{\tau(1)}^{2^{k-1}}u_{\tau(2)}^{2^{k-2}}\cdots u_{\tau(k)}^1\right)^q\cdot u_1^{d+1}u_2^{\alpha_2+d+1}\cdots u_k^{\alpha_k+d+1}.$$ \noindent Viewing $P_U$ as a polynomial in $\mathbb{Z}_2[u_1,\ldots, u_k]/(u_1^{n+1},\ldots, u_k^{n+1})$, a consideration of the exponents $\beta_i$ in each monomial $x_1^{\beta_1}\cdots x_k^{\beta}$ in the resulting expansion of $P_U$ shows that all such terms vanish unless $\sigma(i)=\tau(k-i+1)=i$ for all $i$, in which case the mononmial is $x_1^n\cdots x_k^n$. Thus $P_{U}=u_1^n\cdots u_k^n$. By Proposition~\ref{prop:cohomology}, any $\mathbb{Z}_2^k$-equivariant map  $f\colon Y(n,k)\rightarrow U$ has a zero, and in particular so must any $\mathfrak{S}_k^\pm$-equivariant map $F\colon Y(n,k)\rightarrow U_k^{\oplus m} \oplus V_k$. 
 \end{proof}

\section{Proof of the $t=1$ case of Theorem~\ref{thm:two-radon}}
\label{sec:degree}

We prove the remaining cases of Theorem~\ref{thm:two-radon} over the course of the next two sections. By Theorem~\ref{thm:implications}, this will complete the proof of Theorem~\ref{thm:two-transversal}. Here we prove the $t=1$ case of Theorem~\ref{thm:two-radon}, which for organizational purposes we state as a separate theorem:

\begin{theorem}
\label{thm:degree}
Let $m=2^s +1$, where $s \ge 1$ and $d\ge 1$. Suppose that $n\ge d+\lceil\frac{3m}{2}\rceil+1$. If $\F$ is a family of subsets of $[n+1]$ with $\chi(\KG^{4}(\F)) \le m$, then for any continuous map $f\colon\Delta_n\rightarrow \R^d$, there are two Radon pairs $(\sigma_1^+, \sigma_1^{-}), (\sigma_2^+, \sigma^{-}_2)$ for~$f$ such that each intersection of the form $\sigma_1^\pm \cap \sigma_2^\pm$ does not contain any set from~$\F$. 
\end{theorem}

As the $s=1$ case of Theorem~\ref{thm:degree} is already covered by Theorem~\ref{thm:upper}, we consider $s\geq 2$ throughout the remainder of this section.

We shall restrict the setup of Section~\ref{sec:general} to the case where $k=2$. Thus $\mathfrak{S}_2^\pm=\Z_2^2\rtimes \langle \tau \rangle=D_4$ is the dihedral group of order 8 (that is, the symmetry group of a square), where the transposition $\tau$ acts by swapping the coordinates of any $g=(g_1,g_2)\in \Z_2^2$. In what follows, for any $n\geq 1$ we let $(S^n)^2=Y(n,2)$ and we denote the non-free part $Y'(n,2)$ by $(S^n)^2_{>1}$.

Let $m=2^s+1$ with $s\geq 2$, let $c=\lceil\frac{3m}{2}\rceil$, and let $n = d+c+1$. In order to prove Theorem~\ref{thm:degree}, it will be necessary to use the full dihedral group instead of just the normal $\mathbb{Z}_2^2$-subgroup.  We will prove that any $D_4$-equivariant map $F\colon (S^n)^2\rightarrow U_2^{\oplus m}\oplus V_2$ has a zero, which demonstrates the existence of the desired Radon pair for any $f\colon \Delta_n\rightarrow \R^d$. As before, we proceed by contradiction. Assuming that $F$ never vanishes, a mapping degree argument will lead to a contradiction with the degree calculation of a certain $D_4$-equivariant map $G\colon (S^{c-1})^2\rightarrow U_2^{\oplus m}$ crucially used in~\cite{BFHZ18} in obtaining case (ii) of Theorem~\ref{thm:hyperplane-exact} (see~\cite[Lemma 5.6]{BFHZ18}). We state the properties of this map that we shall need as Lemma~\ref{lem:degree} below. As $c-1=3\cdot 2^{s-1}+1$, observe that  $\dim((S^{c-1})^2)=3\cdot 2^s+2=\dim(S(U_2^{\oplus m}))$ and so the degree of a mapping between $(S^{c-1})^2$ and $S(U_2^{\oplus m})$ is defined after a choice of orientation. We let $\phi\colon U_2^{\oplus m}\setminus\{0\}\rightarrow S(U_2^{\oplus m})$ be the radial projection, which is $D_4$-equivariant.

\begin{lemma}
\label{lem:degree}
Let $s\geq 2$, let $m=2^s+1$, and let $c=\lceil\frac{3m}{2}\rceil$. Then there exists a non-vanishing continuous $D_4$-equivariant map $G\colon (S^{c-1})^2\rightarrow U_2^{\oplus m}$ such that $\phi \circ G \colon (S^{c-1})^2\rightarrow S(U_2^{\oplus m})$ has non-zero degree modulo 8.  
\end{lemma}

\subsection{Outline of the Proof of Theorem~\ref{thm:degree}} 

Before providing the details of the proof Theorem~\ref{thm:degree}, we first provide a summary of our argument. Let $n=d+c+1$, where again $c=\lceil\frac{3m}{2}\rceil$ with $m=2^s+1$ and $s\geq 2$. We have $\dim ((S^{n-1})^2)=\dim S(U_2^{\oplus m}\oplus V_2)$. Assuming that $F\colon (S^n)^2\rightarrow U_2^{\oplus m}\oplus V_2$ never vanishes, the restriction $$\overline{F}\colon (S^{n-1})^2\rightarrow U_2^{\oplus m}\oplus V_2$$ of $F$ is likewise non-vanishing (here we view $S^{n-1}$ as the equatorial sphere inside $S^n$), and composing this with the radial projection $\rho\colon (U_2^{\oplus m}\oplus V_2)\setminus\{0\}\rightarrow S(U_2^{\oplus m}\oplus V_2)$ gives the existence of a $D_4$-equivariant map $$\rho\circ \overline{F}\colon (S^{n-1})^2\rightarrow S(U_2^{\oplus m}\oplus V_2).$$ As this map factors through $(S^n)^2$, it follows from elementary (co)homological considerations that $\rho\circ \overline{F}$ has degree zero.  On the other hand, a generalization of the equivariant Hopf degree theorem shows that the degrees of $\rho\circ \overline{F}$ and $\rho \circ \widetilde{G}$ coincide modulo 8, where $\widetilde{G}\colon (S^{n-1})^2\rightarrow U_2^{\oplus m}\oplus V_2$ is a natural $D_4$-equivariant extension of $G$ which is defined in Section~\ref{sec:5.2}. This is because $\rho\circ \overline{F}$ and $\rho\circ \widetilde{G}$ are $D_4$-equivariantly homotopic on the non-free part $(S^{n-1})^2_{>1}$ of $(S^{n-1})^2$, as shown by  Proposition~\ref{prop:shielding}. 

An easy regular value argument then shows that $\deg(\rho \circ \widetilde{G})=\deg(\phi \circ G)$. This implies that $\deg(\phi\circ G) \equiv 0 \pmod{8}$, contradicting Lemma~\ref{lem:degree} and thereby establishing Theorem~\ref{thm:degree}. 

\subsection{Equivariant Homotopy on the Non-Free Part}
\label{sec:5.2} 
Our degree argument relies on comparing the $D_4$-equivariant map $\overline{F}\colon (S^{n-1})^2\rightarrow U_2^{\oplus m}\oplus V_2$ to a $D_4$-equivariant extension $\widetilde{G}\colon (S^{n-1})^2\rightarrow U_2^{\oplus m}\oplus V_2$ of the $D_4$-equivariant map $G\colon (S^{c-1})^ 2\rightarrow U_2^{\oplus m}$ guaranteed by Lemma~\ref{lem:degree}. 
To that end, let $c\geq 1$ be arbitrary and suppose that  $f\colon (S^c)^2\rightarrow U_2^{\oplus m}$ is a $D_4$-equivariant map which is non-vanishing on the non-free part $(S^c)^2_{>1}$ of~$(S^c)^2$. For any $d\geq 1$ and $n=d+c+1$, we extend $f$ to a $D_4$-equivariant map $$\widetilde{f}\colon (S^n)^2\rightarrow U_2^{\oplus m}\oplus V_2$$ whose zeros are the same as those of $G$, and which in particular does not vanish on $(S^n)^2_{>1}$. Namely, view $S^n$ as the join $S^n=S^c\ast S^d$. Then each $y_i$ from any $y=(y_1,y_2)\in (S^n)^2$ is the formal convex sum $y_i=(1-t_i)y_i^c\oplus t_iy_i^d$ where $y_i^c\in S^c$, $y_i^d\in S^d$, and $t_i\in[0,1]$. Defining \begin{equation}\label{eqn:widetilde{f}}\widetilde{f}(y)=(1-t_1)(1-t_2) f(y_1^c,y_2^c)\oplus (t_1y_1^d, t_2y_2^d),\end{equation} we see that $\widetilde{f}$ is well--defined, $D_4$-equivariant, and has the same zero set as~$f$. Note also that the restriction of $\widetilde{f}$ to $(S^n)^2_{>1}$ equivariantly extends the restriction of $f$ to $(S^c)^2_{>1}$. In what follows, we shall occasionally need to ensure smoothness of the maps $f$ and $\widetilde{f}$. To give $S^n$ the usual smooth structure, $S^c$ is connected to $S^d$ via arcs along great circles, so that   $y_i=\left(\cos(\pi t_i/2)y_i^c,\sin(\pi t_i/2)y_i^d\right)$ for any $y_i\in S^n$. Correspondingly, $\widetilde{f}\colon (S^n)^2\rightarrow U_2^{\oplus m}$ given by \begin{equation}\label{eqn:smoothwidetilde{f}}\widetilde{f}(y)=\cos(\pi t_1/2)\cos(\pi t_2/2)f(y_1^c,y_2^c)\oplus \left(\sin(\pi t_1/2)y_1^d,\sin(\pi t_2/2)y_2^d\right)\end{equation} is then smooth whenever $f$ is.

We now return to the case where $c=\lceil\frac{3m}{2}\rceil$ with $m=2^s+1$ and $s\geq 2$. Let $\overline{F}\colon (S^{n-1})^2\rightarrow U_2^{\oplus m}\oplus V_2$ be the map above and let $G\colon (S^{c-1})^2\rightarrow U_2^{\oplus m}$ be as in Lemma~\ref{lem:degree}. Now consider the restrictions $\overline{F}\mid_{(S^{n-1})^2_{>1}}, \widetilde{G}\mid_{(S^{n-1})^2_{>1}}\colon (S^{n-1})^2_{>1}\rightarrow (U_2^{\oplus m}\oplus V_2)\setminus\{0\}$. the following proposition will be crucial for our degree argument.

\begin{proposition} 
\label{prop:shielding} 
Let $m\geq 1$ and suppose that $n=c+d+1$, where $c=\lceil\frac{3m}{2}\rceil$ and $d\geq 1$. If $f_1, f_2 \colon (S^{n-1})^2_{>1} \rightarrow (U_2^{\oplus m}\oplus V_2)\setminus\{0\}$ are $D_4$-equivariant maps, then $f_1$ and $f_2$ are $D_4$-equivariantly homotopic. In particular, the maps $\overline{F}\mid_{(S^{n-1})^2_{>1}}$ and $\widetilde{G}\mid_{(S^{n-1})^2_{>1}}$ above are $D_4$-equivariantly homotopic. 
\end{proposition}

\begin{proof}[Proof of Proposition~\ref{prop:shielding}]
The non-free part $(S^{n-1})^2_{>1}$ of the $D_4$-action on $(S^{n-1})^2$ consists of the two disjoint $n$-spheres $S^{n-1}_+:=\{(x,x) \ : \ x \in S^{n-1}\}$ and $S^{n-1}_-:=\{(x,-x) \ : \ x \in S^{n-1}\}$. The $D_4$-action on ~${(S^{n-1})^2_{>1}}$ is free when restricted to $\mathbb{Z}_2^2$. Moreover, the standard free $\mathbb{Z}_2$-equivariant cellular structure on $S^{n-1}$ -- i.e., with two $i$-dimensional open hemisphere cells $e_+^i$ and $e_-^i$ for each $0\leq i <n$ -- gives rise to a $D_4$-equivariant CW structure on $(S^{n-1})^2_{>1}$ such that the  $\mathbb{Z}_2^2$ action is free. Indeed, for each $0\leq i<n$, one may define the $i$-dimensional cells of $(S^{n-1})^2_{>1}$ by $(e^i_\pm)_+:=\{(x,x)\mid x\in e^i_\pm\}\subset S_+^{n-1}$ and let
$(e^i_\pm)_-:=\{(x,-x)\mid x\in e^i_\pm\}\subset S_-^{n-1}$. If $\tau$ denotes the generator of $\mathfrak{S}_2$, then $\tau$ acts trivially on $S_{n-1}^+$. On the other hand, $\tau$ acts antipodally on $S_{n-1}^-$, so that $\tau\cdot (x,-x)=g\cdot (x,-x)$ where $g=(1,1)\in \mathbb{Z}_2^2$. Moreover, the CW structure on $(S^{n-1})^2_{>1}$ respects these actions. 

Let $U_2^+$ be the subset of $U_2$ given by $U_2^+=\left\{\sum_{g\in G} r_g g \in U_2 \mid r_{(0,1)}=r_{(1,0)}\right\}$, and likewise let  $U_2^-=\left\{\sum_{g\in G} r_g g\in U_2\mid r_{(0,0)}=r_{(1,1)}\right\}$.  Also let $V_2^+=\{(v_1,v_2)\in (\R^{d+1})^2\mid v_1=v_2\}$  and $V_2^-=\{(v_1,v_2)\in (\R^{d+1})^2\mid v_1=-v_2\}$. Now define $U^+=(U_2^+)^{\oplus m}\oplus V_2^+$ and $U^-=(U_2^-)^{\oplus m} \oplus V_2^-$. Thus $U^+$ is the subset of $U_2^{\oplus m}\oplus V_2$ on which the transposition $\tau$ acts trivially, while and $U^-$ is the subset of $U_2^{\oplus m}\oplus V_2$ on which $\tau$ acts as $g=(1,1)\in \Z_2^2$. Letting $\tau\cdot h$ denote the permutative action of $\tau$ on any $h\in \mathbb{Z}_2^2$, one has $\tau h \tau=(\tau\cdot h)\tau$ for each $h\in \Z_2^2$, where $\tau\cdot h=h$ if $h\in \langle g \rangle$ and $\tau\cdot h =g+h$ otherwise. It is then easily verified that the subset $U:=U^+\cup U^-$ of $U_2^{\oplus m}\oplus V_2$ is $D_4$-invariant, and moreover that $h\cdot U^\pm=U^\pm$ if $h\in \langle g\rangle$ and $h\cdot U^\pm =U^\mp$ otherwise.

Now consider the given $D_4$-equivariant maps  $f_1,f_2\colon {(S^{n-1})^2_{>1}}\rightarrow U_2^{\oplus m}\oplus V_2$. Each of these necessarily maps $S_+^{n-1}$ to $U^+$ and $S_-^{n-1}$ to $U^-$. To obtain the claimed homotopy, extend the $D_4$-action on ${(S^{n-1})^2_{>1}}$ to ${(S^{n-1})^2_{>1}} \times [0,1]$ by letting $D_4$ act trivially on~$[0,1]$. Now let $X= (S^{n-1})^2_{>1} \times \{0,1\}$ and let $h \colon X \to U \setminus\{0\}$ be given by $h=f_1|_{(S^{n-1})^2_{>1}}$ on $(S^{n-1})^2_{>1} \times \{0\}$ and $h=f_2|_{(S^{n-1})^2_{>1}}$ on $(S^{n-1})^2_{>1} \times \{1\}$. The product CW structure on ${(S^{n-1})^2_{>1}} \times [0,1]$ is $D_4$-equivariant, and the action is free when restricted to $\Z_2^2$. It follows from elementary equivariant obstruction theory that there is a $\mathbb{Z}_2^2$-equivariant extension $H\colon {(S^{n-1})^2_{>1}} \times [0,1] \to U\setminus \{0\}$ of $f$ so that $H(S_+^{n-1}\times I)\subset U^+$ and $H(S_-^{n-1}\times I)\subset U^-$. To see this, observe that the boundary of any cell $(e_\pm^i)_\pm \times (0,1)$ has dimension at most $n-1=d+c$, while the connectivity of $U^\pm\setminus \{0\}$ is $2m+d$. As $c=\lceil\frac{3m}{2}\rceil\leq 2m$, we have $n-1\leq 2m+d$. Thus one may define $H$ by induction on the dimension of cells of $(S^{n-1})^2_{>1}\times I$, $\Z_2^2$-equivariantly extending cell-by-cell at any stage: if $H$ is so constructed on the $i$-skeleton, $0\leq i < n-1$, then the composition of $H$ with the attaching map of $(e_+^{i+1})_+\times I$ gives a map from an $i$-sphere to $U^+$ which is therefore nullhomotopic and so extends to this cell. One then defines $H$ on the remaining $(i+1)$-dimensional cells $g\cdot ((e_+^{i+1})_+\times I)$ for $g\in \Z_2^2\setminus\{0\}$ to preserve equivariance, and this guarantees that $H((e_\pm^{i+1})_\pm)\subset U^\pm)$.

The map $H\colon {(S^{n-1})^2_{>1}} \times [0,1]\rightarrow U\setminus\{0\}$ thus given by obstruction theory is a $\Z_2^2$-equivariant homotopy between $f_1$ and $f_2$. However, the action of $\tau$ on $(S_n)^+\times I$ and on $U^+$ is trivial while the action of $\tau$ on $(S_n)^-\times I$ and $U^-$ is given by $g=(1,1)$.  As $H$ sends $S^n_+\times I$ to $U^+$ and $S^n_-\times I$ to~$U^-$, the homotopy is necessarily $D_4$-equivariant. 
\end{proof} 

\begin{remark}
\label{rem:shielding} In Section~\ref{sec:transversal}, it will be important to note that, provided $m\geq 2$, Proposition~\ref{prop:shielding} holds for $n=c+d+1$ and $c=\lceil\frac{3m}{2}\rceil$ when the domain sphere $S^{n-1}$ is replaced by $S^n$. This is because we still have no obstruction in equivariantly extending the map $h\colon (S^n)^2_{>1} \times \{0,1\}\rightarrow U\setminus\{0\}$ to all of $(S^n)^2_{>1}\times I$. Indeed, the dimension of any cell $(e_\pm^i)_\pm \times (0,1)$ is now $n=d+c+1$, and this is still no greater than the connectivity of the codomain when $m\geq 2$. The same dimension count shows that one may also let $d=-1$ in this case (as well as in Proposition~\ref{prop:shielding}), or in other words that these results holds when the module $V_2$ is omitted. This observation will also be needed in Section~\ref{sec:transversal}. 
\end{remark}

\subsection{Degree Calculations}

We now perform the degree calculations outlined at the beginning of this section. As before, let $\overline{F}\colon (S^{n-1})^2\rightarrow U_2^{\oplus m}\oplus V_2$ be the restriction of our map $F\colon (S^n)^2\rightarrow U_2^{\oplus m}\oplus V_2$, which we assume to be non-vanishing, and let $G\colon (S^{c-1})^2\rightarrow U_2^{\oplus m}$ be as in Lemma~\ref{lem:degree}. We consider the maps  $\rho\circ \overline{F},\rho\circ\widetilde{G}\colon (S^{n-1})^2\rightarrow S(U_2^{\oplus m}\oplus V_2),$ where again $\rho\colon (U_2^{\oplus m}\oplus V_2)\setminus\{0\}\rightarrow S(U_2^{\oplus m}\oplus V_2)$ denotes the radial projection. 

First, we show that \begin{equation}\label{eqn:mod} \deg(\rho \circ \overline{F})\equiv\deg(\rho \circ \widetilde{G})\pmod{8}.\end{equation} 

\begin{proof}[Proof of equation~\ref{eqn:mod}] By Proposition~\ref{prop:shielding}, we have that  $\overline{F}_{(S^{n-1})^2_{>1}}$ and $\widetilde{G}_{(S^{n-1})^2_{>1}}$ are $D_4$-equivariantly homotopic, and therefore $(\rho\circ \overline{F})_{(S^{n-1})^2_{>1}}$ and $(\rho\circ \widetilde{G})_{(S^{n-1})^2_{>1}}$ are $D_4$-equivariantly homotopic. Thus equation~(\ref{eqn:mod}) is an immediate consequence of the following generalization of the equivariant Hopf theorem~\cite[Corollary 2.4]{KuBa96}.\end{proof}

\begin{theorem}
\label{thm:hopf}
Let $n\geq 1$ and let $M^n$ be a compact oriented $n$-dimensional manifold with the action of a finite group~$\Gamma$. Let $N\subseteq M$ be a closed $\Gamma$-invariant subset which contains the non-free part of~$M$. Then for any two $\Gamma$-equivariant maps $f_1,f_2\colon M^n \rightarrow S^n$ that are equivariantly homotopic on~$N$, we have that $\deg(f_1) \equiv \deg(f_2) \pmod{|\Gamma|}.$ \end{theorem}

Next, we show that \begin{equation}\label{eqn:extension} \deg(\rho\circ \widetilde{G})=\deg(\phi\circ G), \end{equation} where $\phi\colon U_2^{\oplus m}\setminus\{0\}\rightarrow S(U_2^{\oplus m})$ is the radial projection.  

\begin{proof}[Proof of equation~\ref{eqn:extension}] By the Whitney approximation theorem, $G\colon (S^{c-1})^2\rightarrow U_2^{\oplus m}$ is homotopic to a smooth map, and therefore so are $\phi\circ G$ and $\rho\circ \widetilde{G}$. It therefore suffices to assume that $G$ is smooth. Moreover, $\deg(\rho\circ \widetilde{G})=\deg(\rho \circ (\widetilde{\phi\circ G}))$, and therefore it suffices to assume that the image of $G$ lies in $S(U_2^{\oplus m})$. Indeed, letting $G_s(y)=(1-s)G(y_c)+(1-s)(\phi\circ G)(y_c)$ for each $y_c\in (S^{c-1})^2$ and $s\in [0,1]$, then $\rho\circ \widetilde{G_s}$ gives an explicit smooth homotopy between $\rho\circ \widetilde{G}$ and $\rho\circ (\widetilde{\phi\circ G})$. 

By the explicit formula for $\widetilde{G}$, we have $$\widetilde{G}(y_1,y_2)=\cos(\pi t_1/2)\cos(\pi t_2/2)G(y_1^c,y_2^c)\oplus \left(\sin(\pi t_1/2)y_1^d,\sin(\pi t_2/2)y_2^d\right),$$ where  $(y_i^c,y_i^d)\in S^{c-1}\times S^d$, $t_i\in [0,1]$, and $y_i=\left(\cos( \pi t_i/2)y_i^c, \sin(\pi t_i/2)y_i^d\right)$ for $i\in\{1,2\}$.

Let $u\in S(U_2^{\oplus m})$ be a regular value for $G \colon (S^{c-1})^2\rightarrow S(U_2^{\oplus m})$. Thus $G^{-1}(u)$ is finite, $G$ is a local diffeomorphsim at each $y_c\in G^{-1}(u)$, and $\deg(G)=\sum_{y_c\in G^{-1}(u)}\deg_{G}(y_c)$ where $\deg_{G}(y_c)=\pm 1$ is the local degree (i.e., the sign of the Jacobian determinant) at $y_c$. 

Consider $u\oplus 0\in S(U_2^{\oplus m}\oplus V_2)$. The formula for $\widetilde{G}$ shows that $(\rho\circ \widetilde{G})^{-1}(u\oplus 0)=G^{-1}(u)\oplus 0$. As $\widetilde{G}$ equals $G$ on $(S^{c-1})^2$ and is the identity map on $(S^d)^2$, it is again clear from the formula for $\widetilde{G}$ that $\rho\circ \widetilde{G}$ is a local diffeomorphism at each $y=(y_c,0)$ and that the Jacobian determinants of $\rho\circ \widetilde{G}$ at $y$ and $G$ at $y_c$ are the same. Thus $u\oplus 0$ is a regular value for $\rho \circ \widetilde{G}$ and $\deg(\rho\circ\widetilde{G})=\sum_{y\in (\rho\circ \widetilde{G})^{-1}(u\oplus 0)}\deg_{\rho\circ \widetilde{G}}(y)=\sum_{y_c\in G^{-1}(u)}\deg_G(y_c)=\deg(G)=\deg(\phi\circ G)$. \end{proof}

As our last computation, we show that \begin{equation}\label{eqn:F} \deg(\rho\circ \overline{F})=0.\end{equation} 

\begin{proof}[Proof of equation~\ref{eqn:F}]
We have that $\rho\circ \overline{F}$ factors as $$(S^{n-1})^2\hookrightarrow (S^n)^2\stackrel{\rho\circ F}{\rightarrow} S(U_2^{\oplus m}\oplus V_2).$$ The induced map  $(\rho\circ \overline{F})_{\ast}\colon H_{2n-2}((S^{n-1})^2;\Z)\rightarrow H_{2n-2}(S(U_2^{\oplus m}\oplus V_2);\Z)$ on homology must be the zero map because $H_{2n-2}((S^n)^2;\Z)=0$ by the K\"unneth formula. Thus $\deg(\rho\circ\overline{F})=0$.
\end{proof}

Putting equation~\ref{eqn:F} together with equations~\ref{eqn:mod} and ~\ref{eqn:extension} completes the proof of Theorem~\ref{thm:degree}:

\begin{proof}[Proof of Theorem~\ref{thm:degree}] If the map $F\colon (S^n)^2\rightarrow U_2^{\oplus m}\oplus V_2$ never vanishes, it follows from equations~(\ref{eqn:mod})--(\ref{eqn:F}) that $\deg(\phi\circ G)\equiv 0 \pmod{8}$. This contradicts Lemma~\ref{lem:degree}, so any $F\colon (S^n)^2\rightarrow U_2^{\oplus m}\oplus V_2$ must have a zero and therefore the claimed Radon pair exists.\end{proof}   

\section{Proof of the $t=0$ case of Theorem~\ref{thm:two-radon}}
\label{sec:transversal}

We now prove the $t=0$ case of Theorem~\ref{thm:two-radon}. As before, we state this as a separate theorem:

\begin{theorem}
\label{thm:transversal-bounds}
Let $m=2^s$ for some $s \ge 1$ and $d\ge 1$. Suppose that $n\ge d+ 3\cdot 2^{t-1} +1$. If $\F$ be a family of subsets of $[n+1]$ with $\chi(\KG^{4}(\F)) \le m$, then for any continuous map $f\colon\Delta_n\rightarrow \R^d$, there are two Radon pairs $(\sigma_1^+, \sigma_1^{-}), (\sigma_2^+, \sigma^{-}_2)$ for~$f$ such that each intersection of the form $\sigma_1^\pm \cap \sigma_2^\pm$ does not contain any set from~$\F$. 
\end{theorem}

Let $n = d+c+1$ where $c=3\cdot 2^{s-1}$ with $s\geq 1$. As in Section~\ref{sec:degree}, we shall use contradiction to prove that any $D_4$-equivariant map $F\colon (S^n)^2\rightarrow U_2^{\oplus m}\oplus V_2$ must have a zero, now using a smooth homotopy argument to extend a non-equivariant result  of~\cite{BFHZ16}. 

\subsection{Conversion to the Join Scheme} 
\label{sec:join} 
In order to use computations of ~\cite{BFHZ16}, we first convert from the product scheme used so far to the ``join scheme.'' For any $n\geq 1$, the join $$(S^n)^{\ast 2}=\{z=(1-\mu)y_1\oplus \mu y_2\mid y_1, y_2\in S^n\,\, \text{and}\,\, 0\leq \mu \leq 1\}$$ of $S^n$ consists of all formal convex sums of any two elements from $S^n$. Topologically, $(S^n)^{\ast 2}$ is a sphere of dimension $2n+1$, and the product $(S^n)^2$ is naturally identified with the subset of $(S^n)^{\ast 2}$ consisting of all formal sums for which $\mu=1/2$. As for the product, there is a canonical $D_4$-action on~$(S^n)^{\ast 2}$, with the non-free part~$((S^n)^{\ast 2})_{>1}$ consisting of those $z$ such that either 
(i) $\mu=0$ or $\mu=1$, or else (ii) $\mu=1/2$ and simultaneously $y_1=\pm y_2$. 

Let $W_2$ be the subset of $\R^2$ given by $W_2=\{(r,-r) \ : \ r\in \R\}$, which we view as a $D_4$-module by letting $\mathbb{Z}_2^2$ act trivially and letting $\mathfrak{S}_2$ act by swapping coordinates. For arbitrary $n,d\geq 1$ and any continuous map $f\colon (S^n)^2\rightarrow \R^d$, one may consider the ``join extension'' $$J_0(f)\colon (S^n)^{\ast 2} \rightarrow \R^d \oplus W_2$$ given by \begin{equation}\label{eqn:J_0(f)} J_0(f)(z)= \mu(1-\mu)f(y)\oplus (\mu-1/2,1/2-\mu).\end{equation} 
 Note that the restriction of $J_0(f)$ to  $(S^n)^2$ is precisely~$\frac14 f$. Observe that $J_0(f)(z)=0$ if and only if $\mu=1/2$ and $f(y)=0$, and in particular $J_0(f)$ does not vanish on the non-free part $(S^n)^{\ast 2}_{>1}$ provided $f\colon (S^n)^2\rightarrow \R^d$ does not vanish on $(S^n)^2_{>1}$. Moreover, if $\R^d$ is realized as a $D_4$-module and $f\colon (S^n)^2\rightarrow \R^d$ is $D_4$-equivariant, it is again immediate that $J_0(f)$ is also $D_4$-equivariant.
 
 In what follows smoothness will be important, so let us note here that $(S^n)^{\ast 2}=S^{2n+1}$ has the standard smooth structure if the join is realized by connecting the two copies of $S^n$ via great circles. Explicitly, each $z\in (S^n)^{\ast 2}$ is of the form $z=(\cos(\pi \mu/2) y_1, \sin(\pi \mu/2)y_2)$ for $(y_1,y_2)\in S^n$ and $\mu\in [0,1]$. The non-free part is the same as before. As in Section~\ref{sec:degree}, we view each $S^n=S^c\ast S^d$ as the join, and now realize this join explicitly via arcs on great circles connecting the two sphere factors. Corresponding to this realization of $(S^n)^{\ast 2}$, we will modify the definition of the join map and set \begin{equation}\label{eqn:alt} J(f)(z)= \cos(\pi\mu/2)\sin(\pi\mu/2)f(y)\oplus \left(\cos^2(\pi\mu/2)-1/2,\sin^2(\pi\mu/2)-1/2\right)\end{equation}  for each $z=(\cos(\pi \mu/2) y_1, \sin(\pi \mu/2)y_2)$. As before, $J(f)$ is $D_4$-equivariant if $f$ is.  Again $J(f)(z)=0$ if and only if $\mu=1/2$ and $f(y)=0$ and so the zero sets of $J_0(f)$ and $J(f)$ coincide. In the case that $f\colon (S^n)^2\rightarrow \R^d$ is smooth, equation~\ref{eqn:alt} ensures that $J(f)\colon (S^n)^{\ast 2}\rightarrow \R^d\oplus W_2$ is smooth. Moreover, it follows easily from this formula that $0$ is a regular for $J(f)$ precisely when $0$ is a regular value for $f$. This fact will be needed crucially in what follows.

As in Section~\ref{sec:degree}, we will compare the join $J(F)$ of any $D_4$-equivariant map $F\colon (S^n)^2\rightarrow U_2^{\oplus m}\oplus V_2$ 
to the join $J(\widetilde{G})$ of an extension $\widetilde{G} \colon (S^n)^2\rightarrow U_2^{\oplus m}\oplus V_2$ of a certain $D_4$-equivariant map $G\colon (S^c)^2\rightarrow U_2^{\oplus m}$  which we describe in Section~\ref{sec:previous}. By Proposition~\ref{prop:shielding}, there is a non-vanishing $D_4$-equivariant homotopy between the restriction of $F$ and $\widetilde{G}$ to the non-free part $(S^n)^2_{>1}$. The following guarantees an equivariant homotopy between $J(F)$ and $J(\widetilde{G})$ which is non-vanishing on the non-free part.

\begin{proposition} 
\label{prop:shielding2} 
Let $m\geq 2$ and suppose that $n=c+d+1$, where $c=\lceil\frac{3m}{2}\rceil$ and $d\geq 1$.
\begin{compactenum}[(a)]
\item If $f_1, f_2 \colon (S^n)^2_{>1} \rightarrow (U_2^{\oplus m}\oplus V_2)\setminus\{0\}$ are $D_4$-equivariant maps, then $J(f_1), J(f_2) \colon (S^n)^{\ast 2}_{>1}\rightarrow (U_2^{\oplus m}\oplus W_2\oplus  V_2)\setminus\{0\}$ are $D_4$-equivariantly homotopic. 
\item If $f_1,f_2\colon (S^n)^2\rightarrow U_2^{\oplus m}\oplus V_2$ are $D_4$-equivariant maps which are non-vanishing on $(S^n)^2_{>1}$, then there exists a $D_4$-equivariant homotopy $H\colon (S^n)^{\ast 2}\times I \rightarrow U_2^{\oplus m}\oplus V_2$ from $J(f_1)$ and $J(f_2)$ which is non-vanishing on $(S^n)^{\ast 2}_{>1}\times I$. 
\end{compactenum}
\end{proposition}

\begin{remark}
\label{rem:shielding2}
As with Proposition~\ref{prop:shielding}, it will be important to note that Proposition~\ref{prop:shielding2} also holds when $d=-1$, that is, with the module $V_2$ omitted. 
\end{remark}
 
\begin{proof}[Proof of Proposition~\ref{prop:shielding2}] For the first claim, Remark~\ref{rem:shielding} above shows that there is a $D_4$-equivariant homotopy $H\colon (S^n)^2_{>1}\times I \rightarrow \R^d$ between $f_1$ and $f_2$. Defining $J(H)\colon {(S^n)^{\ast 2}_{>1}}\times I \rightarrow (U_2^{\oplus m}\oplus V_2\oplus W_2) \setminus\{0\}$  by $$J(H)(z,t)=\cos(\pi \mu/2)\sin (\pi \mu/2)H(y,t)\oplus \left(\cos^2(\pi \mu/2)-1/2,\sin^2(\pi \mu/2)-1/2\right)$$  gives the desired $D_4$-equivariant homotopy between $J(f_1)$ and~$J(f_2)$. 

For the second claim, we use elementary equivariant obstruction theory to extend $J(H)$ to an equivariant homotopy $\widehat{J(H)}\colon (S^n)^{\ast 2} \times I\rightarrow U_2^{\oplus m}\oplus V_2$. Namely, given the $D_4$-equivariant structure on~$(S^n)^{\ast 2}$, one has a corresponding $D_4$-equivariant structure on $(S^n)^{\ast 2}\times I$ by letting $D_4$ act trivially on~$I$. Our homotopy is already defined on $(S^n)^2\times\{0,1\}$ as well as the non-free part of $(S^n)^2\times \{0,1\}$, which are both $D_4$-invariant subcomplexes. The desired equivariant extension now exists simply because the codomain $U_2^{\oplus m}\oplus V_2\oplus W_2$ is contractible. 
\end{proof}

Let us outline the argument to follow. As discussed in Sections~\ref{sec:previous} and~\ref{sec:6.3}, the map $G\colon (S^c)^2\rightarrow U_2^{\oplus m}$  will be chosen so that  $J(\widetilde{G})$ is smooth, transverse to zero, and has a zero set $J(\widetilde{G})^{-1}(0)$ consisting of an odd number of full $D_4$-orbits. Using a combination of the Whitney approximation and Thom transversality theorems, we will replace (Section~\ref{sec:6.4}) the continuous equivariant homotopy between $J(\widetilde{G})$ and $J(F)$ guaranteed by Proposition~\ref{prop:shielding2} with a smooth equivariant homotopy between $J(\widetilde{G})$ and some $D_4$-equivariant non-vanishing smooth function. This homotopy will be transverse to 0 and non-vanishing on the non-free part of $(S^n)^{\ast 2}\times I$, from which it will follow that $J(\widetilde{G})$ has an even number of full $D_4$-orbits. This contradiction establishes that any equivariant $F$ must vanish and so completes the proof of Theorem~\ref{thm:transversal-bounds}.

\subsection{Previous Calculations}
\label{sec:previous} 

The existence of the map $G$ will be based on the Borsuk--Ulam type results of ~\cite{BFHZ16} arising from the $m=2^s$ case of the Gr\"unbaum--Hadwiger--Ramos problem. Crucial to these was the existence of a certain $D_4$-equivariant CW--structure on $(S^n)^{\ast 2}$ for any $n\geq 1$.  While we shall not need the details of its construction, we note that it is obtained by intersecting the unit sphere of $(\R^{n+1})^2$ with a cone stratification on $(\R^{n+1})^2$ originally due to Fox--Neuwirth~\cite{FN72} and Bj\"orner--Ziegler~\cite{BZ92}. We refer the interested reader to ~\cite[Section 3]{BFHZ16} for further details. We summarize the key properties are as follows  (see ~\cite[Theorem 3.11]{BFHZ16}). Observe that $\dim (S^n)^{\ast 2}_{>1}=n$, so that none of the top (open) cells of $(S^n)^{\ast 2}$ lie in $(S^n)^{\ast 2}_{>1}$.

\begin{theorem}
\label{thm:CW} 

There is a $D_4$-equivariant CW structure on $(S^n)^{\ast 2}$ such that 

\begin{itemize}

\item $(S^n)^{\ast 2}_{>1}$ is a $D_4$-equivariant subcomplex and the action on the cells of $(S^n)^{\ast 2}\setminus (S^n)^{\ast 2}_{>1}$ is free, and

\item the top-dimensional cells of $(S^n)^{\ast 2}$ consist of a full $D_4$-orbit.

\end{itemize}

\end{theorem} 

Let $m=2^s$ with $s\geq 1$, and let $c=3m/2$. By evaluating measures concentrated on a non-standard moment curve, it was shown in ~\cite[Section 4.3]{BFHZ16} that there is a continuous $D_4$-equivariant map $G_0\colon (S^c)^2\rightarrow U_2^{\oplus m}$ such that $J_0(G_0)\colon (S^c)^{\ast 2}\rightarrow U_2^{\oplus m}\oplus W_2$ the following:  

\begin{itemize}

\item $J_0(G_0)$ does not vanish on the subcomplex $((S^c)^{\ast 2})^{(2c)}$ of cells of codimension one. 
  \item  $J_0(G_0)$ has an odd number of zeros in a generating top-dimensional cell of $(S^n)^{\ast 2}$, and therefore $J_0(G)^{-1}(0)$ consists of an odd number of full $D_4$-orbits.
 \end{itemize}
  
In particular, note that $J_0(G_0)$ does not vanish on the non-free part $(S^n)^{\ast 2}_{>1}$, and therefore $G_0$ does not vanish on $(S^c)^2_{>1}$. As the zero sets of $J_0(G_0)$ and $J(G_0)$ coincide, we conclude that the same facts hold for $J(G_0)$. By the explicit obstruction theory calculations given in~\cite{BFHZ16}, one may state the following:

\begin{proposition}
\label{prop:Borsuk-Ulam}
Let $m=2^s$ and $c=3m/2$, where $s\geq 1$. Suppose that $G\colon (S^c)^2\rightarrow U_2^{\oplus m}$ is a continuous $D_4$-equivariant map such that $J(G)\colon (S^c)^{\ast 2} \rightarrow U_2^{\oplus m}\oplus W_2$ is  non-vanishing on $[(S^c)^{\ast 2}]^{(2c)}$,  $J(G)^{-1}(0)$ is finite, and $J(G)$ is a local homeomorphism at each zero of $J(G)$. Then $|J(G)^{-1}(0)|$ is an odd multiple of 8. 
\end{proposition} 

\begin{proof}[Proof of Proposition~\ref{prop:Borsuk-Ulam}] For $G\colon (S^c)^2\rightarrow U_2^{\oplus m}$ as in the statement of Proposition~\ref{prop:Borsuk-Ulam}, it follows that $G$ is  non-vanishing on $((S^c)^{\ast 2})^{(2c-1)}$ and in particular is non-vanishing on $(S^c)^2_{>1}$. By Remark~\ref{rem:shielding2} following Proposition~\ref{prop:shielding2}, the restrictions of $J(G)$ and $J(G_0)$ to $(S^c)^{\ast 2}_{>1}$ are necessarily $D_4$-equivariantly homotopic via a non-vanishing homotopy.  From here, we appeal to the equivariant obstruction theory computations given in ~\cite[Section 4.3.2]{BFHZ16} which we now summarize.  

First, as $J(G)$ is non-vanishing on $[(S^c)^{\ast 2}]^{(2c)}$, composition with the ($D_4$-equivariant) radial projection $\rho\colon (U_2^{\oplus m}\oplus W_2)\setminus \{0\}\rightarrow S(U_2^{\oplus m}\oplus W_2)$ gives a $D_4$-equivariant map $$\alpha=\rho\circ J(G) \colon [(S^c)^{\ast 2}]^{(2c)}\rightarrow S(U_2^{\oplus m}\oplus W_2).$$ 

Consider the obstruction co-cycle $$\mathfrak{o}(\alpha)\in C^{2c+1}_{D_4}\left([(S^c)^{\ast 2},(S^c)^{\ast 2}_{>1}; \pi_{2c}(S(U_2^{\oplus m}\oplus V_2\oplus W_2))\right).$$ For each top-dimensional cell $e$ of $(S^c)^{\ast 2}$, $$\mathfrak{o}(\alpha)(e)\in \pi_{2c}(S(U_2^{\oplus m}\oplus W_2))\cong \mathbb{Z}$$ is the homotopy class of the composition $\alpha \circ \varphi\colon S^{2c} \rightarrow S(U_2^{\oplus m}\oplus W_2)$ of $\alpha$ with the attaching map  $\varphi\colon S^{2c} \rightarrow [(S^c)^{\ast 2}]^{(2c)}$, and so $\mathfrak{0}(e)$ can therefore be identified with the mapping degree $\deg(\alpha\circ \varphi)$ once a generator is chosen. As $\alpha$ is non-vanishing on $[(S^c)^{\ast 2c}]^{(2c)}$, it is a standard fact of mapping degree theory (see, e.g.,~\cite[Proposition 4.5]{OR09}) that $\deg(\alpha\circ \varphi)$ is the sum of the zeros of $J(G)$ which lay in the (open) cell $e$, where each such zero is counted with signs and multiplicities. As $J(G)$ is a local homeomorphism at each zero of $J(G)$ by assumption, each multiplicity is either $+1$ or $-1$.

Now consider the obstruction class $$[\mathfrak{o}(\alpha)]\in \mathcal{H}^{2c+1}_{D_4}\left((S^c)^{\ast 2},{(S^c)^{\ast 2}_{>1}}); \pi_{2c}(S(U_2^{\oplus m}\oplus W_2))\right).$$ Letting $\alpha_0=\rho\circ J(G_0)\mid_{[(S^c)^{\ast 2}]^{2c}}$, it follows from the $d=-1$ case of Proposition~\ref{prop:shielding2} that $\alpha$ and $\alpha_0$ are $D_4$-equivariantly homotopic, and therefore that  $[\mathfrak{o}(\alpha)]=[\mathfrak{o}(\alpha_0)]$. 

 Let $e$ denote a $D_4$-generator of $\mathcal{C}^{2c+1}_{D_4}\left((S^c)^{\ast 2},{(S^c)^{\ast 2}_{>1}}); \pi_{2c}(S(U_2^{\oplus m}\oplus W_2))\right)$, and let $\zeta$ denote a generator of $\pi_{2c}(S(U_2^{\oplus m}\oplus W_2))\cong \mathbb{Z}$. Thus $\mathfrak{o}(\alpha)=a\cdot \zeta$, where $a=\deg(\rho\circ \alpha_0)$. The explicit computations of ~\cite[Section 4.3.2]{BFHZ16} show that $[\mathfrak{o}(\alpha)]=0$ if and only if $a$ is even. On the other hand, the discussion above shows that number of zeros of $J(G_0)$ in the open cell $e$ is odd, so $\deg(\rho\circ \alpha_0)$ is odd and $[\mathfrak{o}(\alpha)]\neq 0$. This implies that $\deg(\rho \circ \alpha)$ is odd as well, and as this degree is the sum $\sum_{z\in J(G)^{-1}(0)\cap e} \pm 1$ ranging over the zeros of $J(G)$, we conclude that $|J(G)^{-1}(0)\cap e|$ is also odd. Finally, the zeros of $J(G)$ all lie in top-dimensional cells by assumption, and by Theorem~\ref{thm:CW}(b) the top-dimensional cells are a free $D_4$-orbit of the generating cell $e$. By equivariance, each top-dimensional cell has the same number of zeros, and therefore $|J(G)^{-1}(0)|$ is an odd multiple of 8, as claimed. \end{proof}

\subsection{Construction of the maps $G$ and $\widetilde{G}$}
\label{sec:6.3} 

Starting with the $D_4$-equivariant map $G_0\colon (S^c)^2\rightarrow U_2^{\oplus m}$ above, it will follow from an application of Thom transversality that there exists smooth $D_4$-equivariant map $G\colon (S^c)^2\rightarrow U_2^{\oplus m}$ so that $J(G)\colon (S^c)^2\rightarrow U_2^{\oplus m}\oplus W_2$ is transverse to zero and is non-vanishing on $[(S^c)^{\ast 2}]^{(2c)}$. It follows automatically that  $J(G)^{-1}(0)$ is a 0-dimensional manifold, hence is finite, and by the transversality condition that $J(G)$ is a local homeomorphism at each zero of~$J(G)$. By Proposition~\ref{prop:Borsuk-Ulam}, we therefore have that $|J(G)^{-1}(0)|$ is an odd multiple of 8. 

\begin{proposition} 
\label{prop:regular} 

Let $m=2^s$ and $c=3m/2$, where $s\geq 1$. Then there exists a smooth $D_4$-equivariant map $G\colon (S^c)^2\rightarrow U_2^{\oplus m}$ which is transverse to zero such that $J(G)\colon (S^c)^{\ast 2} \rightarrow U_2^{\oplus m}\oplus W_2$ is non-vanishing on $[(S^c)^{\ast 2}]^{(2c)}$ and $J(G)$ is transverse to zero. In particular, $|G^{-1}(0)|$ is an odd multiple of 8. 
\end{proposition}

\begin{proof}[Proof of Proposition~\ref{prop:regular}]
Let $G_0\colon (S^c)^2\rightarrow U_2^{\oplus m}$ be as above. By (equivariant) Whitney approximation~\cite[Theorem 4.2, Chapter 6]{Br72}, for any $\varepsilon>0$ there exists a smooth $D_4$-equivariant map $G^\varepsilon_0$  which is a uniform $\varepsilon$-approximation of $G_0$. It follows immediately from the explicit formula for the join map that $J(G^\varepsilon_0)$ is a uniform $\varepsilon$-approximation of $J(G)$. As each $G_0^\varepsilon$ is smooth, the explicit formula for a join map  shows that each $J(G_0^\varepsilon)$ is smooth as well. By compactness, it follows that $J(G_0^\varepsilon)$ is non-vanishing on $[(S^c)^{\ast 2}]^{(2c)}$ for some $\varepsilon>0$. Let $G_1$ be such a map. In particular, $G_1$ is non-vanishing on the $(2c-1)$ skeleton of $(S^c)^2$.

By an equivariant version~\cite[Proposition 2.2]{LW75} of relative Thom transversality (see, e.g., ~\cite[Theorem 1.35]{Mi58}), for any $\varepsilon >0$ there is a $D_4$-equivariant smooth map $G_1^\varepsilon$ which is a uniform $\varepsilon$-approximation of $G_1$, is transverse to zero, and is equal to $G_1$ on $[(S^c)^2]^{(2c-1)}$.
This is because the $D_4$-action on $X:=(S^c)^2\setminus [((S^c)^2)^{(2c-1)}]$ is free and $G_1$ is automatically transverse to 0 on $[(S^c)^2]^{(2c-1)}$ because $G_1$ is non-vanishing there, so that $G_1$ can be modified without changing its value on the $(2c-1)$-skeleton to obtain a uniform $\varepsilon$-approximation which is transverse to 0. By the formula for join maps above, it follows that each $J(G_1^\varepsilon)$ is transverse to zero as well (and is a uniform $\varepsilon$-approximation of~$J(G_1)$). Finally, by compactness, there exists some $\varepsilon>0$ such that $J(G_1^\varepsilon)$ is non-vanishing on all of~$[(S^c)^{\ast 2}]^{(2c)}$. Letting $G=G_1^\varepsilon$ for some such $\varepsilon$ completes the proof. 
\end{proof}

Now let $d\geq 1$ be arbitrary and let $n=d+c+1$. As in Section~\ref{sec:degree}, we let $S^n=S^c\ast S^d$. For each $y=(y_1,y_2)\in (S^n)^2$ we have $y_i=(\cos(\pi t_i/2)y_i^c, \sin(\pi t_i/2)y_i^d)$ with $y_i^c\in S^c$, $y_i^d\in S^d$, and $t_i\in [0,1]$. For $G\colon (S^c)^2\rightarrow U_2^{\oplus m}$ as given by Proposition~\ref{prop:regular}, we define a smooth $D_4$--equivariant extension
$$\widetilde{G}\colon (S^n)^2\rightarrow U_2^{\oplus m}\oplus V_2$$ by \begin{equation}
   \widetilde{G}(y)=\cos(\pi t_1/2)\cos(\pi t_2/2)~G(y_1^c,y_2^c)\oplus (\sin(\pi t_1/2)y_1^d,\sin(\pi t_2/2)y_2^d). \end{equation} As $G$ is transverse to zero, it follows that $\widetilde{G}$ is transverse to zero as well. Moreover, the zeros of $G$ and $\widetilde{G}$ coincide, so $\widetilde{G}$ is non-vanishing on $[(S^n)^2]^{(2n-1)}$. Considering $J(\widetilde{G})\colon (S^n)^{\ast 2}\rightarrow U_2^{\oplus m}\oplus V_2\oplus W_2$, we likewise conclude that it is transverse to zero and non-vanishing on $[(S^n)^{\ast 2}]^{(2n)}$. Again, the zeros of $J(G)$ and $J(\widetilde{G})$ coincide, so $J(\widetilde{G})^{-1}(0)$ consists of an odd number of free $D_4$-orbits.

\subsection{An Equivariant Homotopy Argument}
\label{sec:6.4} 

Recall that we are assuming that $F\colon (S^n)^2\rightarrow U_2^{\oplus m}\oplus V_2$ is never vanishing, and thus and so is its join extension. In order to make our smooth equivariant homotopy argument, we will need to show the existence of a transverse homotopy between $J(\widetilde{G})$ and a non-vanishing function. As with Proposition~\ref{prop:Borsuk-Ulam}, this follows from an application of Whitney approximation and Thom transversality.

\begin{proposition}
\label{prop:transverse homotopy} If $F\colon (S^n)^2\rightarrow U_2^{\oplus m}\oplus V_2$ is never vanishing, then there exists a smooth homotopy $H\colon (S^n)^{\ast 2}\times I \rightarrow U_2^{\oplus m}\oplus V_2\oplus W_2$ with the following properties:

\begin{itemize}
\item $H$ is transverse to zero,
\item $H$ coincides with $J(\widetilde{G})$ on $(S^n)^{\ast 2} \times \{0\}$, and
\item $H$ is non-vanishing on $(S^n)^{\ast 2}_{>1}\times I$ as well as on $(S^n)^{\ast 2}\times \{1\}$.
\end{itemize}
\end{proposition} 

\begin{proof}[Proof of Proposition~\ref{prop:transverse homotopy}]
Let $W=U_2^{\oplus m}\oplus V_2\oplus W_2$. By Proposition~\ref{prop:shielding2}, there exists a continuous $D_4$-equivariant homotopy $H'\colon (S^n)^{\ast 2}\times I \rightarrow W$ from $J(\widetilde{G})$ to $J(F)$ which is non-vanishing on the non-free part $(S^n)^{\ast 2}_{>1}\times I$. As $F$ is assumed to be never vanishing, so is~$J(F)$. We will now replace this $H'$ by a smooth homotopy which is transverse to zero and coincides with $J(\widetilde{G})$ on $(S^n)^{\ast 2}\times \{0\}$ and is non-vanishing on $(S^n)^{\ast 2}_{>1}\times I$ as well as on $(S^n)^{\ast 2}\times \{1\}$.

First, for any $\varepsilon>0$, relative equivariant Whitney approximation ~\cite[Theorem 4.2]{Br72} guarantees a smooth $D_4$-equivariant map $H'_\varepsilon\colon (S^n)^{\ast 2} \times I \rightarrow W$ which uniformly approximates $H'$ by~$\varepsilon$. Moreover, each $H_\varepsilon$ may be taken to coincide with $J(\widetilde{G})$ on $(S^n)^{\ast 2}\times \{0\}$ because $J(\widetilde{G})$ is $D_4$--equivariant  on the closed subset $(S^n)^{\ast 2}\times \{0\}$. Finally, by compactness  there exists some $\varepsilon>0$ such that $H_\varepsilon$ is non-vanishing on $(S^n)^{\ast 2}_{>1} \times I$ as well as on $(S^n)^{\ast 2} \times \{1\}$.  
Let $H''=H'_\varepsilon$ for such~$\varepsilon$. In particular, $H''$ coincides with $J(\widetilde{G)}$ on the slice $(S^n)^{\ast 2}\times \{0\}$. 

We now replace $H''$ with a smooth equivariant homotopy $H$ which again coincides with $H$ on $(S^n)^{\ast 2}\times \{0\}$, is non-zero on both $(S^n)^{\ast 2}_{>1}\times I$ and $(S^n)^{\ast 2}\times \{1\}$, and is in addition transverse to zero. Namely, as the $D_4$-action on $(S^n)^{\ast 2} \times I$ is free away from $(S^n)^{\ast 2}_{>1}\times I$ and $H''$ is already transverse to zero on $(S^n)^{\ast 2}\times \{0,1\}$ and $(S^n)^{\ast 2}_{>1}\times I$,  by  equivariant Thom transversality ~\cite[Proposition 2.2]{LW75} one can consider any $D_4$-equivariant uniform $H$ of $H''$ which is transverse to zero and so that $H$ coincides with $H'$ on $(S^n)^{\ast 2}\times \{0,1\}$ and $(S^n)^{\ast 2}_{>1}\times I$.
\end{proof} 

Using the $D_4$-equivariant transverse homotopy $H\colon (S^n)^{\ast 2}\times I \rightarrow U_2^{\oplus m}\oplus W_2\oplus V_2$ of Proposition~\ref{prop:transverse homotopy}, we may now prove Theorem~\ref{thm:transversal-bounds} via a standard zero counting argument as in ~\cite{Mi97}. See also ~\cite{HK20,MaSo21}.

\begin{proof}[Proof of Theorem~\ref{thm:transversal-bounds}] Let $W=U_2^{\oplus m}\oplus V_2\oplus W_2$. Assuming that $F\colon (S^n)^2\rightarrow U_2^{\oplus m}\oplus V_2$ is non-vanishing, let $H\colon (S^n)^{\ast 2}\times I \rightarrow W$ be as guaranteed by Proposition~\ref{prop:transverse homotopy}.  As $H$ is transverse to~$0$, as is $H\mid_{(S^n)^{\ast 2}\times \{0,1\}}$, it follows that $H^{-1}(0)$ is a closed one-dimensional manifold with boundary $\partial H^{-1}(0)=H^{-1}(0)\cap ((S^n)^{\ast 2}\times \{0,1\})$. As $H$ is non-vanishing on both $(S^n)^{\ast 2}\times \{1\}$ and $H\mid_{(S^n)^{\ast 2}\times \{0\}}=J(\widetilde{G})$, we conclude that $\partial H^{-1}(0)=J(\widetilde{G})^{-1}(0)$. Moreover, as $H$ is non-vanishing on $(S^n)^{\ast 2}_{>1}\times I$, $H^{-1}(0)$ lies in the free part of $(S^n)^{\ast 2}\times I$. On the other hand, $H^{-1}(0)$ is diffeomorphic to a union of segments and circles. It follows that $D_4$ freely permutes each of the segment components of $H^{-1}(0)$: if some non-trivial $g \in D_4$ set-wise fixed a segment, then a point of the segment must be fixed by $g$ too, but points outside $(S^n)^{\ast 2}_{>1}\times I$ have trivial stabilizer. Thus the number of segment components of $H^{-1}(0)$ is a multiple of $|D_4|$, and therefore $\partial H^{-1}(0)=J(\widetilde{G})^{-1}(0)$ consists of an even number of full $D_4$-orbits. This contradicts Proposition~\ref{prop:regular}, and therefore $F$ must vanish.\end{proof}
	
\section{A generalization of $\Delta(1,3)=3$ via topological reduction} 
\label{sec:reduction} 

For any $m\geq1$, it is an easy consequence of the Ham Sandwich theorem that an upper bound $\Delta(2m,k)\leq c$ implies the upper bound $\Delta(m,k+1)\leq c$ (see, e.g.~\cite{Ha66, Ra96}). In particular, the exact value $\Delta(1,3)=3$ of Theorem~\ref{thm:hyperplane-exact} follows from $\Delta(2,2)=3$. We now derive a similar reduction for multiple Radon pairs which follows from an application of Sarkaria's theorem. Theorem~\ref{thm:three-transversal} is then an immediate corollary. 

\begin{proposition}
\label{prop:topextensions}
Let $c \ge 1$, $d \ge 1$, $k\ge 1$, and $m\ge 1$ be integers. Suppose that the following holds for any $n\ge d+c+2$: \begin{itemize} \item [ ] For any continuous map $f\colon \Delta_{n-1}\to \mathbb{R}^d$ and any family $\F$ of subsets of $[n]$ with $\chi(\KG^{2^k}(\F)) \le 2m$, there are $k$ minimal Radon pairs $(\sigma_1^+,\sigma_1^-),\ldots, (\sigma_k^+,\sigma_k^-)$ for~$f$ such that none of the $2^k$ intersections of the form $\sigma_1^\pm \cap \dots \cap \sigma_k^\pm$ contains a set from~$\F$.\end{itemize} Then the following holds for any $n\ge d+c+2$: \begin{itemize} \item [ ] For any continuous map $f\colon \Delta_{n-1}\to \mathbb{R}^d$ and any family $\F$ of subsets of $[n]$ with $\chi(\KG^{2^{k+1}}(\F)) \le m$, there are $k+1$ minimal Radon pairs $(\sigma_1^+,\sigma_1^-),\ldots, (\sigma_{k+1}^+,\sigma_{k+1}^-)$ for~$f$ such that none of the $2^{k+1}$ intersections of the form $\sigma_1^\pm \cap \dots \cap \sigma_{k+1}^\pm$ contains a set from~$\F$. \end{itemize}
\end{proposition}

\begin{proof}[Proof of Proposition~\ref{prop:topextensions}] 

We first note that the assumption of Proposition~\ref{prop:topextensions} implies that $c\geq m$. This is because we are assuming that for any $n\geq d+c+2$ we have in particular that for any \emph{linear} map $f\colon \Delta_{n-1}\to \mathbb{R}^d$ and any family $\F$ of subsets of $[n]$ with $\chi(\KG^{2^k}(\F)) \le 2m$, there are $k$ minimal Radon pairs $(\sigma_1^+,\sigma_1^-),\ldots, (\sigma_k^+,\sigma_k^-)$ for~$f$ such that none of the $2^k$ intersections of the form $\sigma_1^\pm \cap \dots \cap \sigma_k^\pm$ contains a set from~$\F$. Thus part (1) of Theorem~\ref{thm:implications} is satisfied (with $2m$ replacing $m$), and so by part (3) of Theorem~\ref{thm:implications} we have $c\geq \Delta(2m,k)$. On the other hand, $\Delta(2m,k)\geq \lceil\frac{2m(2^k-1)}{k}\rceil$ by the lower bound of Ramos~\cite{Ra96}, so in particular $c\geq m$.

Now let $n\ge d+c+2$, let $f\colon \Delta_{n-1}\to \mathbb{R}^d$ be continuous, and suppose that $\F$ is a family of subsets of $[n]$ with $\chi(\KG^{2^{k+1}}(\F)) \le m$, say $\F=\F_1\cup \cdots \cup \F_m$ where $\chi(\KG^{2^{k+1}}(\F_i))=1$ for all $i\in [m]$. We show there are there are $k+1$ minimal Radon pairs $(\sigma_1^+,\sigma_1^-),\ldots, (\sigma_{k+1}^+,\sigma_{k+1}^-)$ for~$f$ such that none of the $2^{k+1}$ intersections of the form $\sigma_1^\pm \cap \dots \cap \sigma_{k+1}^\pm$ contains a set from~$\F$. 

In order to apply Theorem~\ref{thm:sarkaria}, for each $i\in [m]$ we define $\G_i$ to be the family of subsets of $[n]$ consisting of all $2^k$-fold unions of pairwise disjoint members of $\F_i$, i.e.  $$\G_i=\left\{\cup_{j=1}^{2^k} \tau_j \mid \tau_j \in \F_i\,\, \text{for all}\,\, j\,\, \text{and}\,\, \tau_{j_1}\cap \tau_{j_2}=\emptyset\,\,\text{for all} \,\, j_1\neq j_2\right\}.$$ As $\chi(\KG^{2^{k+1}}(\F_i))=1$, 
each non-empty $\G_i$ is an intersecting family. Letting $\G=\cup_{i=1}^m \G_i$, we therefore have that $\chi(\KG(\G))\leq m$ and so $\chi(\KG(\G))\leq c$. As in the proof of Proposition~\ref{prop:dolnikov-finite}, let $\Sigma$ denote the simplicial complex on $[n]$ whose minimal non-faces are the sets of ~$\G$. Since the set family $2^{[n]} \setminus \Sigma$ is obtained from $\G$ by including any supersets of elements of~$\G$, we have $\chi(\KG(2^{[n]} \setminus \Sigma) = \chi(\KG(\G))\leq c$. Applying Theorem~\ref{thm:sarkaria} to the restriction $f\mid_{|\Sigma|}\colon |\Sigma| \rightarrow \R^d$, we conclude that there is a minimal Radon pair $\{\sigma_1^+,\sigma_1^-\}$ for $f$ such that neither $\sigma_1^+$ nor $\sigma_1^-$ contains any set from $\G$. Thus the union of any $2^k$ pairwise disjoint elements from any $\F_i$ is not contained in either $\sigma_1^+$ or $\sigma_1^-$.

For each $i\in [m]$, we now define $$\F_i^+=\{\tau\in \F_i \mid \tau \subset \sigma_1^+\}\,\, \text{and}\,\, \F_i^-=\{\tau\in \F_i \mid \tau \subset \sigma_1^-\}.$$ If $\F_i^+\neq\emptyset$, then necessarily $\chi(\KG^{2^k}(\F_i^+))=1$ since otherwise $\sigma_1^+$ would contain the union of some collection of $2^k$ pairwise disjoint elements of $\F_i$. Likewise, we have $\chi(\KG^{2^k}(\F_i^-))\leq 1$ for all $i \in [m]$ as well. Letting $$\F'=\cup_{i=1}^m \F_i^+ \cup\, \cup_{i=1}^m \F_i^-,$$ we therefore have that $\chi(\KG^{2^k}(\F'))\leq 2m$. By the assumption of Proposition~\ref{prop:topextensions}, there must be $k$ minimal Radon pairs $(\sigma_2^+,\sigma_2^-), \ldots, (\sigma_{k+1}^+,\sigma_{k+1}^-)$ for $f$ such that none of the intersections of the form $\sigma_2^\pm\cap \cdots \cap \sigma_{k+1}^\pm$ contains any set of $\F'$. 

To conclude the proof, suppose that $\tau\in \F_i$ for some $i\in [m]$ and suppose for contradiction that $\tau$ were contained in some intersection of the form $\sigma_1^\pm\cap \sigma_2^\pm \cap \cdots\cap  \sigma_{k+1}^\pm$. Since $\tau$ is contained in $\sigma_2^\pm \cap \cdots\cap  \sigma_{k+1}^\pm$, we conclude that $\tau$ is not in either $\F_i^+$ or $\F_i^-$, and therefore that $\tau$ is not contained in either $\sigma_1^+$ or $\sigma_1^-$. But $\tau\subset \sigma_1^\pm\cap \sigma_2^\pm \cap \cdots\cap  \sigma_{k+1}^\pm$, so we have a contradiction. Thus none of the $2^{k+1}$ intersections of the form $\sigma_1^\pm\cap \sigma_2^\pm \cap \cdots\cap  \sigma_{k+1}^\pm$ contains any set from $\F$ and the proof is complete.\end{proof} 

We conclude with the proof of Theorem~\ref{thm:three-transversal}. 

\begin{proof}[Proof of Theorem~\ref{thm:three-transversal}]
Appealing to the $m=2$ case of Theorem~\ref{thm:two-radon}, Proposition~\ref{prop:topextensions} shows that for any continuous map $f\colon \Delta_{d+4}\rightarrow \mathbb{R}^d$ and any family $\F$ of subsets of $[d+5]$ with $\chi(\KG^8(\F))=1$, there are three minimal Radon pairs $(\sigma_1^+,\sigma_1^-), (\sigma_2^+,\sigma_2^-),$ and $(\sigma_3^+,\sigma_3^-)$ for~$f$ such that none of the eight intersections of the form $\sigma_1^\pm \cap \sigma_2^\pm \cap \sigma_3^\pm$ contains a set from~$\F$.  Theorem~\ref{thm:three-transversal} now follows from Theorem~\ref{thm:implications}. \end{proof}

\section*{Acknowledgements}

The authors are grateful to the anonymous reviewers, whose helpful comments improved the clarity and presentation of the manuscript. In particular, we thank them for suggesting the simplified proof of equation~(\ref{eqn:extension}).

\end{document}